\documentclass[11pt]{amsart}

\usepackage[margin=1in]{geometry}

\usepackage{amssymb,amsfonts,amsthm,mathtools}
\usepackage{enumitem}
\usepackage{hyperref}
\usepackage[T1]{fontenc}
\usepackage{verbatim}
\usepackage{xcolor}

\theoremstyle{plain}
\newtheorem{main}{Theorem}

\newtheorem{thm}{Theorem}[section]
\newtheorem{cor}[thm]{Corollary}
\newtheorem{prop}[thm]{Proposition}
\newtheorem{lem}[thm]{Lemma}

\newtheorem{claim}[thm]{Claim}

\theoremstyle{definition}
\newtheorem{defn}[thm]{Definition}
\newtheorem{remark}[thm]{Remark}

\newtheorem{exmp}[thm]{Example}

\theoremstyle{remark}

\newcommand{\C}{\mathbb{C}}
\newcommand{\cD}{\mathcal{D}}

\newcommand{\bF}{\mathbb{F}}
\newcommand{\cG}{\mathcal{G}}
\newcommand{\cH}{\mathcal{H}}

\newcommand{\N}{\mathbb{N}}\newcommand{\cN}{\mathcal{N}}
\newcommand{\cO}{\mathcal{O}}

\newcommand{\R}{\mathbb{R}}\newcommand{\cR}{\mathcal{R}}

\newcommand{\cT}{\mathcal{T}}

\newcommand{\Z}{\mathbb{Z}}

\newcommand{\act}{\curvearrowright}
\newcommand{\rest}{\restriction}
\newcommand{\sub}{\subseteq}

\newcommand{\eps}{\varepsilon}
\newcommand{\G}{\mathcal{G}}
\newcommand{\av}{\text{av}}
\newcommand{\ap}{\text{ap}}
\newcommand{\abs}[1]{\vert #1 \vert}
\newcommand{\norm}[1]{\vert \vert #1 \vert \vert}
\newcommand{\floor}[1]{\lfloor #1 \rfloor}
\newcommand{\ceil}[1]{\lceil #1 \rceil}
\newcommand{\ang}[1]{\langle #1 \rangle}
\newcommand{\bigang}[1]{\Big\langle #1 \Big\rangle}

\newcommand{\out}{\mathrm{out}}
\newcommand{\ins}{\mathrm{in}}
\renewcommand{\d}{\operatorname{d}}
\newcommand{\dmu}{\d\!\mu}

\makeatletter
\let\c@equation\c@thm
\makeatother
\numberwithin{equation}{section}

\begin{document}

\title{Spectral theory for Borel pmp graphs}
\author{Cecelia Higgins}
\address{Department of Mathematics, Rutgers University, 110 Frelinghuysen Road, Piscataway, NJ 08854-8019, USA}
\email{ch799@math.rutgers.edu}
\author{Pieter Spaas}\thanks{PS is supported by the European Union via an ERC grant (QInteract, Grant No 101078107).}
\address{Centre for the Mathematics of Quantum Theory, Department of Mathematical Sciences, University of Copenhagen, Universitetsparken 5, DK-2100 Copenhagen \O, Denmark}
\email{pisp@math.ku.dk}
\author{Alexander Tenenbaum}
\address{Department of Mathematics, University of California, Berkeley, 970 Evans Hall, Berkeley, CA 94720-3840, USA}
\email{tenenbaum@berkeley.edu}

\date{}

\begin{abstract}
    We initiate a systematic study of spectral theory for bounded-degree Borel pmp graphs. Specifically, we study spectral properties of the associated adjacency and Laplacian operators. We start with proving a spectral characterization of approximate measurable bipartiteness. Next, we adapt classical theorems of Wilf and Hoffman to give novel upper and lower bounds on the approximate measurable chromatic number. Using similar techniques, we then show that the approximate measurable chromatic number of a pmp graph generated by $n$ bounded-to-one functions is at most $2n + 1$. Next, concerning matchings, we introduce a measurable version of Tutte's condition and show that a spectral assumption analogous to the one from a classical theorem of Brouwer and Haemers implies this measurable Tutte condition. Finally, we show that, under a uniform integrability assumption, the spectrum is continuous under local-global convergence.
\end{abstract}

\maketitle

\vspace{-1cm}

\section{Introduction}

A major goal of research in descriptive combinatorics is the development of techniques for constructively solving classical graph theoretic problems -- for instance, vertex coloring, edge coloring, and perfect matching -- on infinite graphs. While many classical techniques go through for infinite graphs, the solutions that result are typically non-constructive, dependent on tools such as the axiom of choice. Thus, in descriptive combinatorics, additional definability or measurability constraints are imposed on the solutions, resulting in a deep theory with striking departures from classical combinatorics. For comprehensive surveys, we refer to \cite{km2020,pikhurko2021}.

In this paper, we study the descriptive combinatorics of a certain class of definable graphs using tools from spectral theory. More precisely, given a standard probability space $(X, \mu)$, we are interested in bounded-degree Borel graphs on $X$ that are \emph{probability-measure-preserving} (or \emph{pmp} for short); see Section~\ref{sec:prelim} for the formal definitions. Borel pmp graphs, sometimes referred to as \emph{graphings} in the literature, arise naturally in various situations. For instance, the Schreier graph of a Borel measure-preserving action of a finitely generated group on $(X, \mu)$ is a (bounded-degree) Borel pmp graph. Borel pmp graphs have received considerable attention in recent descriptive combinatorics research; see, for instance, the survey \cite[Chapters~6--9 and~13]{km2020}, as well as the more recent papers \cite{timar2019,thornton2022,bks2022,cmgt2025}. We remark that the more general class of Borel \emph{quasi-pmp} (or \emph{mcp}) graphs has also played an important role in modern research; see, for instance, \cite{tserunyan2022,ctt2022,bpz2024, bcr2024,poulin2024,bell2025,bgjt2025,tt2025}.

The main inspiration for our work derives from the rich area of (finitary) \emph{spectral graph theory}, in which combinatorial problems on finite graphs are solved using linear algebraic techniques. Concretely, given a finite graph, one may associate to it matrices, such as the adjacency matrix and the Laplacian matrix, that encode information about the graph. The spectra of these matrices then provide information about the combinatorial properties of the associated graph. For an overview of spectral graph theory, we refer to the surveys \cite{bh2012,spielman2025}. 

In the infinitary setting, Mohar initiated the study of the spectra of countably infinite graphs in the 1980s \cite{mohar1982}; see also the survey \cite{mw1989}. Uncountable graphs were first considered much later by Kechris and Tsankov, who studied the spectra of a certain class of Borel graphs \cite{kt2008}. Specifically, they used spectral properties of Schreier graphs of generalized Bernoulli actions of countable discrete groups to characterize properties such as amenability of the associated group actions. Later, Conley and Kechris studied the more general class of Schreier graphs of free measure-preserving actions of countable discrete groups on standard probability spaces. \cite{ck2013}. In particular, they used spectral tools to derive several new results in descriptive combinatorics. These included an approximate measurable version of Hoffman's spectral lower bound on the chromatic number as well as spectral bounds on the approximate measurable chromatic numbers of the graphs $\G_{\bF_n}$, where $\G_{\bF_n}$ denotes the Schreier graph of the Bernoulli shift of $\bF_n$. Thornton later showed that the approximate measurable chromatic number of $\G_{\bF_2}$ is exactly $3$ \cite{thornton2024}.

Further connections between descriptive combinatorics and spectral graph theory have since been established. Of particular note are the recent applications of the expander mixing lemma (see \cite{gv2025} and \cite[Chapter~6]{lyons2024}) and the role of spectral theory in the study of graph limits (see, for instance, \cite{ckt2012, HLS14, Ku16, KT19}, as well as the references in Section~\ref{sec:examples}).

In this paper, we contribute to the literature on bounded-degree Borel pmp graphs through a systematic study of the spectra of the associated adjacency and Laplacian operators; we then use spectral information to prove novel theorems in measurable combinatorics. Our results include a spectral characterization of approximate measurable bipartiteness, spectral upper and lower bounds on the approximate measurable chromatic number, and a spectral inequality that implies a measurable version of Tutte's classical condition characterizing graphs that admit perfect matchings.

Concretely, given a bounded-degree Borel pmp graph $\cG$ on $(X,\mu)$, we define the \emph{adjacency operator} $T_\G: L^2(X)\to L^2(X)$ by 
$$
(T_\G f)(x) = \sum_{y\sim x} f(y)
$$
for all $f\in L^2(X)$, $x\in X$; here we use the notation $y\sim x$ to indicate that $y$ is adjacent to $x$. The second operator we study is the \emph{Laplacian operator} $L_\G: L^2(X)\to L^2(X)$, given by $L_\G = D_\G - T_\G$, where $D_\G\in L^\infty(X)$ is the \emph{degree function} given by $D_\G(x) = \deg(x)$ for all $x\in X$. That is, the Laplacian operator satisfies
$$
(L_\G f)(x) = \deg(x)f(x) - \sum_{y\sim x} f(y)
$$
for all $f\in L^2(X)$, $x\in X$. Both of these operators have been studied already in the literature. In particular, it is a well-known fact that they are bounded and self-adjoint, thus allowing for the application of various tools from operator theory. We give a self-contained proof of this fact in Section~\ref{sec:TLoperators} using \emph{Borel transports} (see Subsection~\ref{subsec:transport}), which are applied frequently throughout. Subsequently, after establishing some first properties of $T_\G$ and $L_\G$, we extend the classical Perron-Frobenius theory, sometimes under the additional assumption of \emph{spectral gap} (see Subsection~\ref{subsec: spectral gap}). Although many classical results generalize to our setting, the proofs often require new ideas. For example, for finite graphs, all spectral values are eigenvalues, and many arguments in finitary spectral graph theory involve the associated eigenvectors. For infinite graphs, however, spectral values are not always eigenvalues; nevertheless, because our operators are self-adjoint, each spectral value admits a sequence of \emph{approximate} eigenfunctions, allowing for various approximation arguments.

Our first main result concerns bipartiteness, which we study in Section~\ref{sec:bipartite}. It is well-known in the classical setting (see, for instance, \cite[Chapter~4]{spielman2025}) that if a finite graph is bipartite, then the spectrum of its adjacency matrix is symmetric about $0$. If the graph is in addition connected, then a strong version of the converse holds: If the negative of the maximum eigenvalue of the adjacency matrix is also an eigenvalue, then the graph is bipartite. The combination of these results therefore gives a complete spectral characterization of bipartiteness for finite connected graphs. We adapt this characterization to the setting of bounded-degree Borel pmp graphs as follows:

\begin{main}[{Theorems~\ref{thm:bipartiteimpliessymspec} and \ref{thm:-dspecimpliesbipartite}}]
\label{thm:bipartite_intro}
    Let $\G$ be a bounded-degree Borel pmp graph on a standard probability space. 
    \begin{enumerate}
        \item If $\G$ is approximately measurably bipartite, then the spectrum of $T_\G$ is symmetric about $0$.
        \item Assume in addition that $\G$ is $d$-regular and ergodic and that $T_\G$ has spectral gap. If $-d$ belongs to the spectrum of $T_\G$, then $\G$ is approximately measurably bipartite.
    \end{enumerate}
\end{main}

While some of the ideas of the classical proof can be used, additional difficulties arise in the measurable setting. For item (1) above, the approximate measurable bipartition only gives a bipartition up to a specified $\eps>0$, and the adjacency operator $T_\G$ mixes the leftover ``bad'' set of measure $\eps$ with the sets of the bipartition. Nevertheless, given $\lambda$ in the spectrum of $T_\G$, the standard idea of swapping the signs of an eigenvector on one part of the bipartition to show that $-\lambda$ is in the spectrum extends to our setting. Indeed, provided that the error is carefully controlled, we can run a similar argument using approximate eigenfunctions and an approximate bipartition instead.  
For item (2), analogously to the classical case, one would like to use the signs of a given approximate eigenfunction of $-d$ to define an approximate bipartition. In order for this to work in the measurable setting, one needs the additional fact that approximate eigenfunctions of $-d$ are approximately constant in absolute value, which is false in general (see Remark~\ref{rmk:nonstronglyergodic}). However, we show that this holds under the assumption of spectral gap, allowing us to complete the argument.

In Section~\ref{sec:upperbounds}, we prove two novel upper bounds on the approximate measurable chromatic number of a bounded-degree Borel pmp graph. Our first bound generalizes a classical theorem of Wilf \cite{wilf1967}, which gives an upper bound on the chromatic number of a finite graph in terms of the maximum eigenvalue of its adjacency matrix.

\begin{main}[{Theorem~\ref{thm: wilf pmp}}]
\label{thm:wilf_intro}
    Let $\G$ be a bounded-degree Borel pmp graph on a standard probability space. Then the approximate measurable chromatic number of $\G$ is bounded above by $\floor{M(T_{\G})} + 1$, where $M(T_{\G})$ is the maximum spectral value of $T_{\G}$.
\end{main}

In the notation of the above theorem, $M(T_{\G})$ can alternatively be characterized as the spectral radius of $T_\G$, and it is bounded above by the maximum degree of $\G$ (see Proposition~\ref{prop:boundedsa}). Therefore, Theorem~\ref{thm:wilf_intro} provides at least as tight an upper bound on the approximate measurable chromatic number as the standard degree plus one bound (see \cite[Proposition~4.6]{kst1999}). In general, when $\G$ is non-regular, Theorem~\ref{thm:wilf_intro} gives an upper bound much tighter than the degree plus one bound and even tighter than the stronger upper bound from Brooks's theorem (see \cite[Theorem~1.2]{cmt2016}) when it is applicable. For example, it is easy to show that $M(T_{\G})=\sqrt{ab}$ for an $(a,b)$-biregular graph (see Corollary~\ref{cor:bireg}), allowing quadratic improvement over the bound given by Brooks's theorem. We however note that biregular graphs are measurably bipartite, and hence trivially admit measurable 2-colorings. It's also worth pointing out that quadratic improvement is the best possible, since we always have $M(T_\G)\geq \sqrt{d}$ where $d$ is the maximum degree (see Lemma~\ref{lem:avdeg}).

For the proof of Theorem~\ref{thm:wilf_intro}, we recursively construct a sequence of disjoint subsets exhausting the vertex space $X$ by collecting all vertices with sufficiently small degree at each step. This depends on the observation that the average degree of $\G$ is bounded above by $M(T_\G)$ (see Lemma~\ref{lem:avdeg}). The remainder of the proof then involves a list-coloring algorithm, which we use to color a $(1-\eps)$-proportion of the graph by backtracking through the sequence of vertex sets.

Next, we consider pmp graphs generated by finitely many bounded-to-one Borel functions. Questions about the descriptive combinatorics of graphs generated by functions originate in the work \cite{kst1999} of Kechris, Solecki, and Todor\v{c}evi\'{c}, who asked whether the Borel chromatic number of a graph generated by $n$ Borel functions is bounded above by $2n + 1$ in case it is finite.\footnote{It is known that the Borel chromatic number may be infinite; see, for instance, \cite[Example~3.2]{kst1999}.} Kechris and Marks later asked whether the \emph{measurable} chromatic number of such a graph is bounded above by $2n + 1$ \cite{km2020}. Quadratic upper bounds on the measurable chromatic number were discovered independently by Miller \cite{miller2008} and by Palamourdas \cite{palamourdas2012}. In \cite{cm2016}, Conley and Miller use a toast construction to give an upper bound of $4n + 1$ on the measurable chromatic number when the graph is Borel hyperfinite. When the generating functions are pairwise commuting, Meehan and Palamourdas obtain the optimal upper bound of $2n + 1$ on the approximate measurable chromatic number in \cite{mp2021}. 

For our next main result, we prove this optimal upper bound on the approximate measurable chromatic number of a graph generated by finitely many bounded-to-one Borel functions in the case when the graph is pmp. While our bound is not spectral in nature, we include it in the present paper because its proof depends on both the backwards list-coloring trick we use to prove Theorem~\ref{thm:wilf_intro} and the application of Borel transports.

\begin{main}[{Theorem~\ref{thm:nfunctions}}]
\label{thm:functions_intro}
    Let $\G$ be a pmp graph on a standard probability space that is generated by $n$ bounded-to-one Borel functions. Then the approximate measurable chromatic number of $\G$ is bounded above by $2n + 1$.
\end{main}

It in fact follows that, for any bounded-degree Borel directed graph $\cD$ whose underlying undirected graph $\G$ is pmp, if each vertex of $\cD$ has out-degree at most $n$, then the approximate measurable chromatic number of $\G$ is at most $2n + 1$; see Remark~\ref{rem: digraph vs functions}.

Another fundamental result in finitary spectral graph theory is Hoffman's theorem \cite{hoffman1970}, which gives an upper bound on the \emph{independence number} of a regular graph in terms of the maximum and minimum eigenvalues of its adjacency matrix (or equivalently, by regularity, its Laplacian matrix). Since the monochromatic sets in a coloring are independent, this yields a lower bound on the chromatic number as well, which in fact applies even to non-regular graphs. In Section~\ref{sec:lowerbounds}, we generalize Hoffman's bound to the measurable setting. Our result also generalizes \cite[Proposition~4.16]{ck2013}, which gives a similar bound for the (regular) Schreier graph of a free measure-preserving action of a finitely generated discrete group.

\begin{main}[{Theorem~\ref{thm: hoffman bounded}}]
\label{thm:hoffman_intro}
    Let $\G$ be a bounded-degree Borel pmp graph on a standard probability space. Then the approximate measurable chromatic number of $\G$ is bounded below by $\ceil{1 - \frac{M(T_{\G})}{m(T_{\G})}}$, where $M(T_{\G})$ and $m(T_{\G})$ are the maximum and minimum spectral values of $T_\G$, respectively.
\end{main}

We note that the lower bound from the theorem is always at least $2$ (unless $\G$ has no edges, in which case $T_\G$ is the zero operator). Indeed, as in the classical case, the minimum and maximum spectral values always satisfy $m(T_\G) < 0 < M(T_\G)$ (see Lemma~\ref{lem:mnegative}).

When $\G$ is $d$-regular, the proof of Theorem~\ref{thm:hoffman_intro} is straightforward and similar to the proof of \cite[Proposition~4.16]{ck2013}. In this case, given an independent set $S$, by applying the Laplacian operator to the function $1_S - \mu(S)1$, an easy computation yields the upper bound
$$
i_\mu(\G) \leq 1 - \frac{d}{M(L_\G)}
$$
on the measurable independence number $i_\mu(\G)$, where $M(L_\G)$ is the maximum spectral value of the Laplacian operator $L_\G$. Since $d$-regularity implies that $L_\G = dI - T_\G$, the desired lower bound on the approximate measurable chromatic number follows. In the non-regular case, we prove the lower bound from Theorem~\ref{thm:hoffman_intro} directly. Given that $\chi_\mu^\ap(\G) = k$, we can for any $\eps>0$ find pairwise disjoint independent sets $A_1, \ldots, A_k\sub X$ such that $B= X\setminus (A_1\cup\ldots\cup A_k)$ satisfies $\mu(B) < \eps$. The proof then proceeds by considering the block decomposition of $T_\G$ with respect to this partition of $X$. Specifically, we prove an inequality on the spectral values of $T_\G$ in terms of the spectral values of the diagonal blocks from this decomposition (see Lemma~\ref{lem: block decomp ineq}); this involves ideas similar to those in the proof of \cite[Theorem~19.5.1]{spielman2025}. We then prove bounds on the spectral values of the blocks corresponding to the ``bad'' set $B$ (compare with the proof of Theorem~\ref{thm:bipartite_intro}).

In Section~\ref{sec:matchings}, we study matchings. In the classical setting, Hall's famous theorem \cite{hall1935} provides a characterization of the existence of a perfect matching in a bipartite graph. For locally finite Borel pmp graphs, Lyons and Nazarov show in \cite{ln2011} that a bipartite (i.e., odd-cycle-free\footnote{We remark that there are odd-cycle-free Borel graphs that fail to be approximately measurably bipartite; see, for instance, \cite{ck2013}, where it is shown that the shift action of the free group $\bF_n$ for $n\geq 2$ has measurable independence number strictly less than $1/2$, thus yielding such examples.}) graph admits a measurable perfect matching under the assumption that it is \emph{strictly expanding} (see Section~\ref{sec:matchings} for the relevant definitions). Further results on measurable perfect matchings have since been obtained for specific classes of Borel pmp graphs (see for instance \cite{bks2022}).

For finite but not necessarily bipartite graphs, a classical theorem of Brouwer and Haemers \cite{BH05} provides a sufficient condition for the existence of a perfect matching in terms of the eigenvalues of the Laplacian matrix $L_\G$. Namely, if $0=\lambda_1 \leq \lambda_2\leq\ldots \leq \lambda_n$ are the eigenvalues of $L_\G$, then $\G$ admits a perfect matching if $n$ is even and $2\lambda_2 \geq \lambda_n$. Note that the inequality $2\lambda_2 \geq \lambda_n$ can be interpreted as a quantitative lower bound on the spectral gap. The proof in \cite{BH05} relies on Tutte's celebrated classical characterization \cite{tutte1947} of the existence of a perfect matching, which is a generalization of Hall's theorem to non-bipartite graphs. Unfortunately, unlike with Hall's theorem, we do not know whether there is a measurable version of Tutte's theorem for strictly expanding graphs.\footnote{In \cite{kun2024}, Kun constructs a regular acyclic measurably bipartite Borel pmp graph having no measurable perfect matching, demonstrating the necessity of the expansion condition.} Nevertheless, in Section~\ref{sec:matchings}, we propose a measurable (expanding) version of Tutte's condition that we call the \emph{strict measurable Tutte condition} (see Definition~\ref{defn: tutte}), and we prove that a spectral inequality analogous to the classical Brouwer--Haemers condition is sufficient to imply this measurable Tutte's condition. We use the notation $L^2_0(X) = L^2(X)\ominus \C 1$.

\begin{main}[{Theorem~\ref{thm: brouwer haemers}}]
\label{thm:matching_intro}
    Let $\G$ be an ergodic regular Borel pmp graph on a standard probability space. Let
    \begin{align*}
        m_L & = \inf_{0 \neq f \in L^2_0(X)} \frac{\ang{L_\G f, f}}{\ang{f, f}}, \quad M_L = \sup_{0 \neq f \in L^2_0(X)} \frac{\ang{L_\G f, f}}{\ang{f, f}}.
    \end{align*}
    If $2m_L \geq M_L$, then the strict measurable Tutte condition holds for $\G$.
\end{main}

While the main structure of the proof of Theorem~\ref{thm:matching_intro} is similar to that of the proof in \cite{BH05}, the computations become considerably more involved and several new ideas are needed. We defer the details to Section~\ref{sec:matchings}.

Finally, in Section~\ref{sec:examples}, we discuss some classes of examples of bounded-degree Borel pmp graphs for which (properties of) their spectra can be computed exactly. These examples arise as \emph{local-global limits} of sequences of finite graphs. Local-global convergence was first introduced and studied in \cite{BR11, HLS14} and has since garnered considerable interest. We will show that in many instances, one can compute the spectrum of a graph that arises as a local-global limit when given the spectra of the graphs that converge to it (see Theorem~\ref{thm: locglob spectrum}). This result yields a variety of graphs with prescribed spectra.

\subsection*{Organization of the paper} Besides the Introduction, this paper consists of seven other sections. In Section~\ref{sec:prelim}, we establish some notation and definitions, and we prove several elementary lemmas to be used throughout the paper. In Section~\ref{sec:TLoperators}, we define the adjacency and Laplacian operators and deduce some initial properties. In Section~\ref{sec:bipartite}, we study bipartiteness and prove Theorem~\ref{thm:bipartite_intro}. In Sections~\ref{sec:upperbounds} and \ref{sec:lowerbounds}, we investigate upper and lower bounds, respectively, on the (approximate) measurable chromatic number, and we prove Theorems~\ref{thm:wilf_intro}, \ref{thm:functions_intro}, and~\ref{thm:hoffman_intro}. In Section~\ref{sec:matchings}, we study matchings and prove Theorem~\ref{thm:matching_intro}. Finally, in Section~\ref{sec:examples}, we consider continuity of the spectrum with respect to local-global convergence and consider some (classes of) examples.

\subsection*{Acknowledgments} We thank Andrew Marks and Riley Thornton for helpful discussions and in particular for helping us to clarify some of the material in Section~\ref{sec:examples}. We also thank Anton Bernshteyn for pointing out an improvement on Lemma~\ref{lem: list coloring from r} (see Remark~\ref{rem: anton}). Furthermore, we thank Filippo Calderoni and Andrew Marks for providing useful comments on a first draft of the~paper.

\section{Preliminaries}\label{sec:prelim}

\subsection{Borel graphs and descriptive combinatorics}\label{subsec:BorelGraphs}

Throughout, let $(X, \mu)$ be a standard probability space, and let $\G = (X, E(\G))$ be a \emph{Borel graph} on $X$, i.e., a (simple, undirected) graph on $X$ such that $E(\G) \sub X^2$ is Borel. We write $x \sim y$ when $x$ and $y$ are adjacent in $\G$. The \emph{degree} of a vertex $x \in X$, denoted $\deg_{\G}(x)$ or simply $\deg(x)$ when $\G$ is clear, is the cardinality of the set $\{y \in X : y \sim x \}$. We say that $\G$ is \emph{locally finite} if $\deg(x)$ is finite for each $x \in X$, that $\G$ has \emph{bounded degree} if there is $d \in \N$ such that $\deg(x) \leq d$ for each $x \in X$, and that $\G$ is \emph{$d$-regular} if there is $d \in \N$ such that $\deg(x) = d$ for each $x \in X$. Note that since we work throughout in the measurable setting, we always implicitly discard sets of measure zero; for instance, we say $\G$ is locally finite if it is locally finite on a conull set. The \emph{minimum degree} of $\G$ is given by $\min\{d : \mu(\{x \in X : \deg(x) = d \}) > 0\}$. If $\G$ has bounded degree, then its \emph{maximum degree} is given by $\max\{d:\mu(\{x\in X: \deg(x) = d\}) >0\}$. 

For a subset $A\sub X$, we write $\cN(A) = \{x\in X : x \text{ has a neighbor in $A$}\}$; note that $\cN(A)$ could include vertices from $A$ if they are adjacent to other vertices in $A$. We also define the \textit{closed neighborhood} (or closed $1$-ball) $B_1(A)$ of $A$ as $B_1(A) = A\cup \cN(A) = \{x\in X: d(x,A)\leq 1\}$. A subset $A\sub X$ is called \emph{$\G$-invariant} if $A=B_1(A)$, i.e., $A$ contains all its neighbors. Finally, we note that, given a subset $A\sub X$, we may consider the \emph{subgraph $\G \rest A$ of $\G$ induced by $A$}, i.e., the graph on $A$ whose edge relation $E(\G \rest A)$ is defined by
$$E(\G \rest A) = E(\G) \cap A^2.$$
It is easy to see that, when $A \sub X$ is Borel, then $\G \rest A$ is Borel. Additionally, we often wish to consider subgraphs induced by measurable sets $A \sub X$ \emph{that are not necessarily Borel}, in which case the set $A$ (endowed with the subspace topology) may fail to be a standard Borel space. However, it is well-known that, if $A$ is measurable, then $A$ is the union of a Borel set $B$ with a null set. Since all our results concern the measurable setting where null sets may be discarded, we write $\cG \rest A$ in this situation for convenience, but in actuality we implicitly consider $\cG \rest B$.

A set $A \sub X$ is \emph{($\G$-)independent} if, whenever $x, y \in A$ are distinct, $x \not \sim y$. Let $Z$ be a set; a map $c : X \to Z$ is a \emph{($Z$-)coloring} of $\G$ if $c^{-1}(\{z\})$ is $\G$-independent for each $z \in Z$. The \emph{chromatic number} of $\G$, denoted $\chi(\G)$, is the least cardinal $\kappa$ for which there is a $Z$-coloring of $\G$ for some set $Z$ of cardinality $\kappa$. If $\vert Z \vert = k \in \N$ and $c$ is a $Z$-coloring of $\G$, we refer to $c$ as a \emph{$k$-coloring} of $\G$.

As discussed in the introduction, in the descriptive combinatorics setting, definability or measurability restrictions are typically placed on colorings. Specifically, the \emph{Borel chromatic number} of $\G$, denoted $\chi_B(\G)$, is the least $k \in \N$ for which there is a Borel $k$-coloring of $\G$ (if such a $k$ exists); furthermore, the \emph{$\mu$-measurable chromatic number} of $\G$, denoted $\chi_{\mu}(\G)$, is the least $k \in \N$ for which there is a $\mu$-measurable $k$-coloring of $\G$ (if such a $k$ exists). Finally, the \emph{approximate $\mu$-measurable chromatic number} of $\G$, denoted $\chi_{\mu}^{\ap}(\G)$, is the least $k \in \N$ for which, for each $\eps > 0$, there is a Borel set $A \sub X$ with $\mu(A) > 1 - \eps$ such that $\G \rest A$ admits a $\mu$-measurable $k$-coloring (if such a $k$ exists). 

Clearly 
$$\chi(\G) \leq \chi_{\mu}(\G) \leq \chi_B(\G) \text{ and } \chi_{\mu}^{\ap}(\G) \leq \chi_{\mu}(\G).$$
We refer to \cite{km2020} for additional results.

We write $E_{\G}$ for the \emph{connectedness relation} of $\G$, i.e., the equivalence relation whose classes are the connected components of $\G$. When $\G$ is locally finite, $E_{\G}$ is a \emph{countable Borel equivalence relation}, i.e., a Borel equivalence relation whose classes are countable. We usually write $x\mathrel{E}_{\G}y$ instead of $(x,y)\in E_{\G}$. We denote by $[x]_{E_{\G}}$ the equivalence class of $x$ and by $[A]_{E_{\G}}$ the saturation of $A\sub X$, i.e., the union of all the classes of points in $A$. 

By considering this associated Borel equivalence relation, we can make sense of various properties for $\G$ from the theory of Borel equivalence relations. For instance, we say $\G$ is \emph{ergodic} if $E_{\G}$ is ergodic. Additionally, the graphs under consideration here are graphs that interact nicely with the probability measure $\mu$ on $X$. We say that the graph $\G$ is \emph{($\mu$-)probability measure preserving}, or \emph{pmp} for short, if $E_{\G}$ is pmp, i.e., whenever $\theta$ is a partial Borel bijection on $X$ with $\text{graph}(\theta) \sub E_{\G}$, $\theta$ preserves $\mu$ in the sense that for any $\mu$-measurable set $A \sub X$, $\mu(A) = \mu(\theta^{-1}(A))$ (equivalently, since $\theta$ is injective, $\mu(A) = \mu(\theta(A))$). Equivalently, $\G$ is pmp if $E(\G)$ can be generated by countably many measure-preserving Borel involutions.
As for subgraphs, we note that, if $\G$ is pmp, then any subgraph of $\G$ induced by a Borel set is also pmp with respect to the renormalized measure.

We finish this subsection by recording the easy observation that \textit{invariance} of sets as defined above for graphs coincides with the usual notion of invariant sets for Borel equivalence relations. We will use this fact without mention.

\begin{lem}
    A measurable set $A\sub X$ is $\G$-invariant if and only if $A$ is $E_\G$-invariant. In particular, $\G$ is ergodic if and only if every $\G$-invariant set $A$ satisfies $\mu(A)\in\{0,1\}$.
\end{lem}
\begin{proof}
First, we note that if $A\sub X$ is $E_\G$-invariant, i.e., $A=[A]_{E_{\G}}$, then since $A\sub B_1(A)\sub [A]_{E_{\G}}$, we immediately have $A=B_1(A)$. In other words $E_\G$-invariant sets are $\G$-invariant.

Conversely, assume $A\sub X$ is $\G$-invariant, i.e., $A=B_1(A)$. An easy induction argument on the distance to $A$ then shows that, whenever $x \in X$ and $d(x, A) < \infty$, we have $x \in A$. Since $d(x,A)<\infty$ if and only if $x\in [A]_{E_{\G}}$, we conclude that $\G$-invariant sets are $E_\G$-invariant.
\end{proof}

\subsection{Bounded operators on Hilbert spaces}\label{subsec:operators}

In this paper, we study certain operators associated to bounded-degree Borel pmp graphs. In this subsection we collect some basic facts that we use throughout; these can be found in any standard reference on bounded operators on Hilbert spaces.

Let $H$ be a \emph{complex Hilbert space}, i.e., a complex vector space equipped with an inner product $\langle \cdot,\cdot\rangle: H\times H \to \C$ such that $H$ is complete for the norm given by $\norm{\xi} = \sqrt{\langle \xi,\xi\rangle}$ for all $\xi\in H$. We denote by $B(H)$ the algebra of \emph{bounded operators} on $H$, i.e., the linear operators with finite operator norm, where the \textit{operator norm} of $T\in B(H)$ is defined by $\norm{T} = \sup\{\norm{T\xi} : \xi\in H, \norm{\xi} = 1\}$.\footnote{Equivalently, it is well-known that these are exactly the operators that are continuous as functions on $H$ equipped with the norm-topology coming from the inner product.} For $T\in B(H)$, we denote by $T^*\in B(H)$ its \textit{adjoint}, i.e., the operator satisfying $\langle T\xi,\eta\rangle = \langle \xi, T^*\eta\rangle$ for all $\xi, \eta \in H$.

We start with establishing some notation and terminology regarding bounded operators, as well as recalling standard facts. For $T\in B(H)$, the \textit{spectrum} of $T$, denoted $\sigma(T)$, is the set
$$
\sigma(T) = \{\lambda \in \C : T - \lambda I \text{ is not invertible}\}.
$$
It is well-known that the set $\sigma(T)$ is always a non-empty compact subset of $\C$. The \textit{spectral radius} of $T$, denoted $\rho(T)$, is defined by
$$
\rho(T) = \sup\{\abs{\lambda} : \lambda \in \sigma(T) \} = \max\{\abs{\lambda} : \lambda \in \sigma(T) \}.
$$
An operator $T\in B(H)$ is \textit{self-adjoint} if $T = T^*$. Most of the operators we encounter are self-adjoint, and we collect some useful properties about their spectra next.

\begin{lem}
    Let $T\in B(H)$ be a self-adjoint operator. Then the following hold:
    \begin{enumerate}[itemsep=3pt]
        \item $\sigma(T)\sub \R$.
        \item $\lambda\in \sigma(T)$ if and only if there exists a sequence $(\xi_n)_{n \in \N}$ of unit vectors in $H$ such that $$\norm{T\xi_n - \lambda\xi_n}\to 0.\footnote{In other words, $\lambda$ is an approximate eigenvalue. Note that this includes the special case when $\lambda$ is an eigenvalue, in which case we can take the constant sequence $\xi_n=\xi$ where $\xi$ is an eigenvector for $\lambda$. In operator theory terms, this means that the spectrum is the union of the discrete and continuous spectra, i.e., the residual spectrum is empty.}$$
    \end{enumerate}
    Denoting $M(T) = \sup_{f \neq 0} \frac{\ang{Tf, f}}{\ang{f, f}}$ and $m(T) = \inf_{f \neq 0} \frac{\ang{Tf, f}}{\ang{f, f}}$, the following also hold:
    \begin{enumerate}[itemsep=3pt]
        \setcounter{enumi}{2}
        \item $\{m(T), M(T)\}\sub \sigma(T) \sub [m(T), M(T)]$.
        \item $M(T) = \sup_{f, \norm{f} = 1} \langle Tf,f\rangle$ and $m(T) = \inf_{f, \norm{f} = 1} \langle Tf,f\rangle$.
        \item $\rho(T) = \max\{\abs{m(T)}, \abs{M(T)}\} = \sup_{f\neq 0} \frac{\abs{\ang{Tf,f}}}{\ang{f,f}}$.
    \end{enumerate}
\end{lem}

The following is also well-known:

\begin{prop}\label{prop:isolated}
    Let $\lambda$ be an isolated point in the spectrum of a self-adjoint operator $T$ on a Hilbert space $H$. Then:
    \begin{enumerate}
        \item $\lambda$ is an eigenvalue of $T$.
        \item There exists $c>0$ such that $\norm{(T-\lambda I)f}\geq c\norm{f}$ for all $f\in \ker(T-\lambda I)^\perp$.
    \end{enumerate}
    In particular, it follows from (2) that $\lambda\notin \sigma(T|_{\ker(T-\lambda I)^\perp})$.
\end{prop}

\begin{cor}\label{cor:largestisolated}
    Let $T\in B(H)$ be a self-adjoint operator. Assume $\sigma(T)\sub (-\infty, \lambda']\cup \{\lambda\}$ for some real numbers $\lambda' < \lambda$. Then 
    $$
    \sup_{0\neq f\in \ker(T-\lambda I)^\perp} \frac{\langle Tf,f\rangle}{\langle f,f\rangle} \leq \lambda'.
    $$
\end{cor}

Although interlacing of eigenvalues as in for instance \cite{Hae95} does not make sense in general in the infinite-dimensional setting, we have the following special case, which does go through and suffices for our purposes (see Proposition~\ref{prop:XYZ}):

\begin{lem}\label{lem:interlacing}
    Let $T\in B(H)$ be a self-adjoint operator, let $K$ be a Hilbert space, and let $S:K\to H$ be an isometry. Assume the following:
    \begin{enumerate}
        \item $\sigma(T)\sub (-\infty,\kappa_2]\cup \{\kappa_1\}$ for some $\kappa_2<\kappa_1$,
        \item $\sigma(S^*TS) \sub (-\infty,\lambda_3]\cup \{\lambda_2\}\cup \{\lambda_1\}$ for some $\lambda_3 < \lambda_2 < \lambda_1$, and
        \item there is a vector $v\in H$ such that the eigenspace associated to $\kappa_1$ for $T$ is $\C v$.
    \end{enumerate}
    Then $\lambda_2 \leq \kappa_2$.
\end{lem}
\begin{proof}
We write $\tilde T=S^*TS$. Let $\tilde v_1,\tilde v_2\in K$ be eigenvectors for $\tilde T$ for $\lambda_1$ and $\lambda_2$, respectively. Take any nonzero vector $\tilde w\in \mathrm{span}\{\tilde v_1, \tilde v_2\} \cap (\C S^*v)^\perp$. Then
$$
\langle S\tilde w, v\rangle = \langle \tilde w, S^*v\rangle = 0.
$$
In particular, it follows from Corollary~\ref{cor:largestisolated} that 
$$
\frac{\langle TS\tilde w, S\tilde w\rangle}{\langle S\tilde w, S\tilde w\rangle} \leq \kappa_2.
$$
Furthermore, since $\tilde w\in \mathrm{span}\{\tilde v_1,\tilde v_2\}$, we have  
$$
\frac{\langle TS\tilde w, S\tilde w\rangle}{\langle S\tilde w, S\tilde w\rangle} = \frac{\langle \tilde T\tilde w, \tilde w\rangle}{\langle \tilde w, \tilde w\rangle} \geq \lambda_2.
$$
This finishes the proof.
\end{proof}

Before continuing, we state one more basic result about the spectrum of an anti-diagonal self-adjoint operator for later reference.

\begin{lem}\label{lem:antidiagspectrum}
    Assume $T\in B(H)$ is a self-adjoint operator. Then the spectrum $\sigma(A)$ of the operator
    $$
    A = \begin{pmatrix}
        0 & T\\
        T & 0
    \end{pmatrix}
    $$
    on $B(H\oplus H)$ is equal to $\sigma(T)\cup -\sigma(T)$.
\end{lem}
\begin{proof}
If $\lambda \in \sigma(T)$ and $\norm{Tf_n - \lambda f_n}\to 0$, then it is immediate that $\norm{A(f_n\oplus f_n) - \lambda (f_n\oplus f_n)}\to 0$ and $\norm{A(-f_n\oplus f_n) + \lambda (-f_n\oplus f_n)}\to 0$. This proves one inclusion. For the other direction, assume $\lambda\in \sigma(A)$, and assume $(f_n\oplus g_n)_{n \in \N}$ is a sequence of approximate eigenfunctions for $\lambda$. Then 
$$
\norm{Tf_n - \lambda g_n}\to 0 \quad \text{and} \quad \norm{Tg_n - \lambda f_n}\to 0.
$$
In particular, it follows that 
$$
\norm{T^2f_n - \lambda^2 f_n} \to 0.
$$
In other words, $\lambda^2\in \sigma(T^2)$, meaning that $T^2 - \lambda^2 I$ is not invertible. Since $T^2 - \lambda^2 I = (T-\lambda I)(T + \lambda I)$, we conclude that either $\lambda\in \sigma(T)$ or $-\lambda\in \sigma(T)$ (or both). This finishes the proof.
\end{proof}

We end this subsection by establishing some basic facts regarding direct sum decompositions. Given an operator $T\in B(H)$ and a decomposition $H = \oplus_{i=1}^n H_i$ of $H$ into a direct sum of Hilbert subspaces, we can consider the corresponding \emph{block decomposition} $(T_{ij})_{i,j=1}^n$ of $T$. In other words, $T_{ij}\in B(H_j,H_i)$ for every $1\leq i,j\leq n$ and given $x = \sum_{i=1}^n x_i$ with $x_i\in H_i$, we have
$$
Tx = \sum_{i=1}^n \sum_{j=1}^n T_{ij}(x_j).
$$
Note that this is just matrix multiplication with respect to the decomposition of $H$:
$$
Tx = \begin{pmatrix}
    T_{11} & T_{12} & \cdots & T_{1n} \\
    T_{21} & T_{22} & \cdots & T_{2n} \\
    \vdots & \vdots & \ddots & \vdots \\
    T_{n1} & T_{n2} & \cdots & T_{nn}
\end{pmatrix}
\begin{pmatrix}
    x_1\\
    x_2\\
    \vdots \\
    x_n
\end{pmatrix} = \begin{pmatrix}
    \sum_{j=1}^n T_{1j}(x_j)\\
    \sum_{j=1}^n T_{2j}(x_j)\\
    \vdots \\
    \sum_{j=1}^n T_{nj}(x_j)
\end{pmatrix}
$$
Furthermore, if $T$ is self-adjoint, then it is easy to see that for every $1\leq i,j\leq n$, $T_{ij}^* = T_{ji}$. In particular, for every $1\leq i\leq n$, $T_{ii}$ is self-adjoint.

\begin{prop}\label{prop: block decomp T}
    Using the above notation, let $T\in B(H)$ be a self-adjoint operator. Then we have the following inequalities:
    \begin{enumerate}
        \item $\forall 1\leq i\leq n$: $M(T_{ii})\leq M(T)$ and $m(T_{ii})\geq m(T)$, i.e., $\sigma(T_{ii})\sub [m(T),M(T)]$.
        \item $\forall 1\leq i,j\leq n$: $\norm{T_{ij}}\leq \norm{T}$.
    \end{enumerate}
\end{prop}
\begin{proof}
We start with proving (1). Fix $1\leq i\leq n$, and let $f\in H_i$. By definition, there exists $g\in H\ominus H_i$, namely $g = \sum_{k\neq i} T_{ki}(f)$, such that 
$$
Tf = T_{ii}f + g.
$$
Hence 
$$
\ang{Tf, f} = \ang{T_{ii}f + g, f} = \ang{T_{ii}f, f} + \ang{g, f} = \ang{T_{ii}f, f},
$$
since $f\perp g$. In particular, we find
\begin{align*}
M(T_{ii}) &= \sup_{f\in H_i, f\neq 0} \frac{\langle T_{ii}f,f\rangle}{\langle f,f\rangle}\\
&= \sup_{f\in H_i, f\neq 0} \frac{\langle Tf,f\rangle}{\langle f,f\rangle}\\
&\leq \sup_{f\in H, f\neq 0} \frac{\langle Tf,f\rangle}{\langle f,f\rangle}\\
&= M(T).
\end{align*}
Similarly, upon replacing $\sup$ by $\inf$, we get $m(T_{ii})\geq m(T)$, finishing the proof of (1).

We now prove (2). Fix $1\leq i,j\leq n$, and let $f\in H_j$. By definition, $\norm{Tf} \leq \norm{T}\norm{f}$. Hence by the Pythagorean theorem, we have
\begin{align*}
\sum_{k=1}^n \norm{T_{kj}f}^2 &= \norm{\sum_{k=1}^n T_{kj}f}^2\\
&= \norm{Tf}^2\\
&\leq \norm{T}^2\norm{f}^2.
\end{align*}
In particular, $\norm{T_{ij}f} \leq \norm{T}\norm{f}$. Since this holds for every $f\in H_j$, the result follows.
\end{proof}

\subsection{Mass transport for pmp graphs}\label{subsec:transport}

In this final preliminary section, we collect several useful results that we use throughout the paper. Fix a standard probability space $(X, \mu)$ throughout. First, we state the following standard measure-theoretic fact, with a proof provided for completeness:

\begin{lem}
\label{lem: dct}
    Let $f : X \to [0, \infty)$ be a $\mu$-measurable function with $\int_X f(x)\dmu(x) < \infty$. Then for all $\eps > 0$, there is $\delta > 0$ such that, whenever $B \sub X$ is $\mu$-measurable, 
    $$\mu(B) < \delta \implies \int_B f(x)\dmu(x) < \eps.$$
\end{lem}

\begin{proof}
    Assume for contradiction that there is $\eps > 0$ such that, for all $\delta > 0$, there is a measurable set $B$ with $\mu(B) < \delta$ but $\int_B f(x)\dmu(x) \geq \eps$. For each $n \in \N$, let $\delta_n = \frac{1}{2^n}$, let $B_n$ be a $\mu$-measurable set such that $\mu(B_n) < \delta_n$ but $\int_{B_n} f(x)\dmu(x) \geq \eps$, and let $f_n = f \restriction B_n$, so that $f_n \leq f$. Note that $\sum_{n \in \N} \mu(B_n) < \infty$, so that, by the Borel--Cantelli lemma, $\mu(\limsup_{n \to \infty} B_n) = 0$. Note that
    $$\limsup_{n \to \infty} B_n = \{x \in X : x \in B_n \text{ for infinitely many $n$} \}.$$
    Therefore, there is a conull set $A \sub X$ such that $\lim_{n \to \infty} f_n(x) = 0$ for all $x \in A$. By the dominated convergence theorem, we thus get $\lim_{n \to \infty} \int_X f_n(x)\dmu(x) = 0$. But then there is $N \in \N$ such that, for all $n > N$, $\int_X f_n(x)\dmu(x) = \int_{B_n} f(x)\dmu(x) < \eps$, a contradiction.
\end{proof}

In the remainder of this subsection, we derive several useful corollaries from the theory of mass transport. We remark that applications of mass transport are ubiquitous in descriptive set theory (see, for instance, \cite{km2004}) and other areas; the underlying ideas trace back to the seminal work of Feldman and Moore \cite{fm1977}, who introduced the left and right counting measures and proved that their equality is equivalent to $\mu$-invariance. 

\begin{defn}
    Let $E$ be a countable Borel equivalence relation on $X$. A \textit{Borel transport} in $E$ is a Borel function $\varphi: E\to [0,\infty)$. Given a Borel transport $\varphi$ in $E$, we define the functions $\out\,\varphi, \ins\,\varphi: X\to [0,\infty)$ by
    $$
        \out\,\varphi(x) = \sum_{yEx} \varphi(x,y), \quad \text{and}\quad 
        \ins\,\varphi(x) = \sum_{yEx} \varphi(y,x)
    $$
\end{defn}

The well-known \emph{mass transport principle} demonstrates that, for any Borel transport $\varphi$, the integral over $\out\,\varphi$ matches the integral over $\ins\,\varphi$. For the formulation below, see for instance \cite[Proposition~5.3]{tserunyan2022}.

\begin{prop}
\label{prop:Anush}
    Let $E$ be a Borel pmp equivalence relation on $(X,\mu)$, and let $\varphi: E\to [0,\infty)$ be a Borel transport in $E$. Then
    $$
    \int_X \out\,\varphi \dmu = \int_X \ins\,\varphi \dmu
    $$
\end{prop}

A first useful corollary is the following:

\begin{cor}\label{cor:neighbors}
    Let $\G$ be a bounded-degree Borel pmp graph on $(X,\mu)$ with maximum degree equal to $d$, and let $A\sub X$ be a measurable subset of $X$. Then
    $$
    \mu(\cN(A)) \leq d\mu(A).
    $$
\end{cor}
\begin{proof}
Define the Borel transport $\varphi_A: E_{\G}\to [0,\infty)$ by $\varphi_A(x,y) = 1_A(y)$ if $x\sim y$ and $\varphi_A(x,y) = 0$ otherwise. Then we can easily compute
$$
\out\,\varphi_A(x) = \sum_{y\sim x} 1_A(y) = \abs{\{y\sim x : y\in A\}},
$$
and 
\begin{align*}
\ins\,\varphi(x) = \sum_{y\sim x} 1_A(x) = 
\begin{cases}
    \deg(x), &\text{ if } x\in A,\\
    0, &\text{ otherwise}.
\end{cases}
\end{align*}
Note that, for any $x \in X$, $\abs{\{y\sim x : y \in A\}} \geq 1$ if and only if $x \in \cN(A)$. Proposition~\ref{prop:Anush} thus implies that 
\begin{align}
\label{ineq: n(a)}
d\mu(A) \geq \int_A \deg(x)\dmu(x) = \int_X \abs{\{y\sim x : y\in A\}} \dmu(x) \geq \mu(\cN(A)),
\end{align}
finishing the proof.
\end{proof}

The next result also follows easily from Proposition~\ref{prop:Anush}.

\begin{cor}\label{cor:transports}
    Let $\G$ be a Borel pmp graph on $(X,\mu)$. For any positive measurable functions $f,f':X\to [0,\infty)$, we have the following equalities:
    \begin{enumerate}
        \item ${\displaystyle \int_X \sum_{y\sim x} f(y) \dmu(x) = \int_X \deg(x)f(x) \dmu(x)}$. 
        \item ${\displaystyle \int_X \sum_{y \sim x} f(y)f'(x) \dmu(x) = \int_X \sum_{y \sim x} f(x)f'(y) \dmu(x)}$.
    \end{enumerate}
\end{cor}
\begin{proof}
We start with (1). Define the Borel transport $\varphi_f: E_{\G}\to [0,\infty )$ by $\varphi_f(x,y) = f(y)$ if $x\sim y$ and $\varphi_f(x,y) = 0$ otherwise. Note that 
    $$
    \mathrm{out}\,\varphi_f(x) = \sum_{y\sim x} \varphi_f(x,y) = \sum_{y\sim x} f(y),
    $$
    and 
    $$
    \mathrm{in}\,\varphi_f(x) = \sum_{y\sim x} \varphi_f(y,x) = \sum_{y\sim x} f(x) = \deg(x)f(x).
    $$
Proposition~\ref{prop:Anush} now immediately implies (1). 

For (2), define the Borel transport $\varphi_{f,f'} : E_{\G} \to [0, \infty)$ by $\varphi_{f,f'}(x, y) = f(y)f'(x)$ if $x \sim y$ and $\varphi_{f,f'}(x, y) = 0$ otherwise. Note that, for any $x \in X$,
    $$
    \mathrm{out}\,\varphi_{f,f'}(x) = \sum_{y \sim x} f(y)f'(x) \quad\text{and}\quad \mathrm{in}\,\varphi_{f,f'}(x) = \sum_{y \sim x} f(x)f'(y).
    $$
Using Proposition~\ref{prop:Anush} again, (2) follows, finishing the proof.
\end{proof}

Finally, we have the following technical inequality:

\begin{cor}\label{cor:ineq}
    Let $\G$ be a Borel pmp graph on $(X,\mu)$, let $A\sub X$ be a measurable subset of $X$, and suppose $f:X\to [0,\infty)$ is a positive measurable function. Then
    $$
    \int_A \sum_{y\sim x} f(y) \dmu(x) = \int_X \abs{\{y\sim x : y\in A\}}\cdot f(x) \dmu(x) \leq \int_{\cN(A)} \deg(x) f(x) \dmu(x).
    $$
\end{cor}
\begin{proof}
Define the Borel transport $\varphi_{A,f}: E_{\G}\to [0,\infty)$ by 
$$
\varphi_{A,f}(x,y) = 1_A(x) f(y)
$$
We then compute
\begin{align*}
\out\,\varphi_{A,f}(x) = \sum_{y\sim x} 1_A(x) f(y) = 
\begin{cases}
    \sum_{y\sim x} f(y), &\text{ if } x\in A,\\
    0, &\text{ otherwise},
\end{cases}
\end{align*}
and
$$
\ins\,\varphi_{A,f}(x) = \sum_{y\sim x} 1_A(y) f(x) = \abs{\{y\sim x : y\in A\}}\cdot f(x).
$$
Applying Proposition~\ref{prop:Anush} once again finishes the proof.
\end{proof}

\section{Operators associated to Borel graphs}\label{sec:TLoperators}

In this section, we define the two main operators we study, namely the adjacency operator and the Laplacian operator. We record some basic properties of these operators, as well as information about their spectral theory and how it connects to properties of the associated Borel graphs. 

Throughout this section, let $(X, \mu)$ be a standard probability space, and let $\G$ be a bounded-degree Borel pmp graph on $X$ with maximum degree $d$.

\subsection{The adjacency operator}

The \textit{adjacency operator} $T_{\G}\in B(L^2(X))$ is defined by
$$
(T_{\G}f)(x) = \sum_{y \sim x} f(y)
$$
for all $f \in L^2(X), x \in X$. 

\begin{exmp}\label{exmp:adjacencymatrix}
    For a finite graph $\G$ with $n$ vertices, one can identify $L^2(X)$ with $\C^n$ equipped with its usual inner product. In this case, the adjacency operator is merely the adjacency matrix (viewed as a linear operator on $\C^n$). Most of the results in this paper can be interpreted in this setting, recovering known facts from the spectral theory for finite graphs.
\end{exmp}

We begin with the following crucial fact:

\begin{prop}\label{prop:boundedsa}
    The adjacency operator $T_{\G}$ is bounded and self-adjoint. Moreover, the operator norm of $T_{\G}$ is bounded by the maximum degree $d$ of $\G$.
\end{prop}

\begin{proof}
    For notational simplicity, we write $T = T_{\G}$ throughout. We first observe that for any function $f\in L^2(X)$ and $x \in X$, the Cauchy--Schwarz inequality gives
    $$
    \abs{Tf(x)}^2 = \bigg|\sum_{y \sim x} f(y)\bigg|^2 \leq \deg(x) \sum_{y \sim x} \abs{f(y)}^2.
    $$

    Using this together with Corollary~\ref{cor:transports}(1) and the fact that $\deg(x) \leq d$ for each $x \in X$, we compute
    \begin{align}\label{eq:Tbdd}
    \begin{split}
        \norm{Tf}^2 & = \int_X \abs{Tf(x)}^2 \dmu(x)\\
        & \leq d \int_X \sum_{y\sim x} \abs{f(y)}^2 \dmu(x)\\
        & = d \int_X \deg(x) \abs{f(x)}^2 \dmu(x)\\
        & \leq d^2 \norm{f}^2.
    \end{split}
    \end{align}
    Since this holds for any $f\in L^2(X)$, we conclude that $\norm{T}\leq d$.

    For self-adjointness, we observe that, for any $f, f' \in L^2(X)$,
    $$
    \langle Tf, f'\rangle = \int_X Tf(x)\overline{f'(x)}\dmu(x) = \int_X \sum_{y \sim x} f(y) \overline{f'(x)}\dmu(x)
    $$
    and similarly 
    $$
    \langle f,Tf'\rangle = \int_X \sum_{y \sim x} f(x)\overline{f'(y)} \dmu(x).
    $$
    Since $T$ is bounded by the first part of the proof, the integrals appearing in the above two equations are finite by the Cauchy-Schwarz inequality. By writing the functions $f,f'$ as linear combinations of positive functions, the desired result thus follows from Corollary~\ref{cor:transports}(2).
\end{proof}

\begin{cor}
    If $\G$ is $d$-regular, then $\norm{T_\G} = M(T_\G) = d$. Moreover, $d$ is an eigenvalue of $T$.
\end{cor}
\begin{proof}
    Let $g : X \to \C$ be the constant function defined by $g(x) = 1$ for all $x \in X$. Then we see that 
    $$\ang{g, g} = \int_X \abs{g(x)}^2\dmu(x) = 1,$$
    and
    $$
    T_{\G}g(x) = \sum_{y \sim x} g(y) = \deg(x) = d \text{\, for all $x \in X$.}
    $$
    We conclude that $g$ is an eigenfunction with eigenvalue $d$, and in particular $\norm{T_\G},M(T_\G)\geq d$. Since Proposition~\ref{prop:boundedsa} implies that $\norm{T_\G},M(T_\G)\leq d$, this concludes the proof.
\end{proof}

For the next result, recall that for positive integers $a$ and $b$, a Borel graph $\cG$ is called \textit{Borel $(a,b)$-biregular} if there exists a partition $X=A\cup B$ into independent Borel sets such that all vertices in $A$ have degree $a$ and all vertices in $B$ have degree $b$.

\begin{cor}\label{cor:bireg}
    If $\G$ is Borel $(a,b)$-biregular, then $\norm{T_\G} = M(T_\G) = \sqrt{ab}$. Moreover, $\sqrt{ab}$ is an eigenvalue of $T_\G$.
\end{cor}
\begin{proof}
    By assumption, we can bipartition $X=A\cup B$ into independent Borel sets $A$ and $B$ containing the vertices of degree $a$ and $b$ respectively. By considering the integrals over $A$ and $B$ separately, a similar computation as in Proposition~\ref{prop:boundedsa} then yields the desired upper bound; furthermore, it is easy to check that the piecewise constant function taking value $\sqrt{a}$ on $A$ and $\sqrt{b}$ on $B$ is an eigenfunction for $\sqrt{ab}$.
\end{proof}

\begin{exmp}\label{ex:irrationalrotation}
    \begin{itemize}
        \item Let $\alpha$ be an irrational number, and consider the graph $\G$ on $X = [0, 1)$ such that $x \sim y$ if and only if $x + \alpha \equiv_1 y$ or $y + \alpha \equiv_1 x$. Then $\G$ is a $2$-regular Borel $\mu$-pmp graph, where $\mu$ is the usual Lebesgue measure. In particular, $\norm{T_{\G}} = M(T_{\G}) = 2$, and $\sigma(T_{\G})\sub [-2,2]$. We show in Example~\ref{ex:irrationalrotation2} using bipartiteness that $-2 \in \sigma(T_{\G})$. Later, in Example~\ref{exmp:rotationspec}, we show using local-global convergence that $\sigma(T_{\G})$ is in fact exactly the interval $[-2, 2]$.
        \item Given any $d$-regular Borel pmp graph, we can always construct a $(2,d)$-biregular graph by ``splitting the edges'' as follows. Let $d \in \N$, let $(X, \mu)$ be a standard probability space, and let $\G = (X, E)$ be a $d$-regular Borel pmp graph. Here we view $E$ as a set of unordered pairs; in particular, every edge in the original (undirected) graph $\G$ is represented by a single point in $E$. Define a measure $\nu$ on $E$ by setting
         $$
         \nu(S) = \frac{1}{2} \int_X \abs{\{x' \in X : \{x, x'\} \in S \}}\dmu(x).
         $$
         We define a graph $\cH$ on the space $Y = X \sqcup E$ as follows. For any $x \in X$ and $e \in E$, we put $x \sim e$ if and only if $e$ is incident on $x$ in $\G$. Note that the sets $X$ and $E$ are $\cH$-independent and that, for any $x, x' \in X$, there is a path between $x$ and $x'$ in $\G$ if and only if there is a path between $x$ and $x'$ in $\cH$. Define now $\eta = \frac{1}{d + 1}(\mu + \nu)$ and note that $\eta$ is a Borel probability measure on $Y$. A straightforward but tedious exercise then demonstrates that $\cH$ is $\eta$-pmp. Moreover, it follows by construction that $Y=X\cup E$ yields the desired Borel partition into the independent sets of vertices with degree $d$ and $2$ respectively. In particular, the previous corollary implies that $M(T_\cH) = \sqrt{2d}$. Since the graph is measurably bipartite, Theorem~\ref{thm:bipartiteimpliessymspec} moreover implies that $m(T_\cH) = -\sqrt{2d}$ (which in this case can also easily be checked by exhibiting a concrete eigenfunction).
    \end{itemize}
\end{exmp}

When $\G$ is not regular, it is often useful to consider the average degree.

\begin{defn}
The \emph{average degree} of $\G$, denoted $\deg_{\av}(\G)$, is defined by
$$
\deg_{\av}(\G) = \int_X \deg(x)\dmu(x).
$$
\end{defn}

In the non-regular setting, as in the finite case, we have the following inequalities relating the maximum spectral value $M(T_\G)$, the average degree, and the maximum degree:

\begin{lem}\label{lem:avdeg}
    Let $d$ be the maximum degree of $\G$. Then the following hold:
    \begin{enumerate}
        \item $m(T_\G) \leq \deg_{\av}(\G) \leq M(T_{\G})$.
        \item $\sqrt{d}\leq M(T_\G) \leq d$.
    \end{enumerate}
\end{lem}
\begin{proof}
(1) Let $g : X \to \C$ again be defined by $g(x) = 1$ for all $x \in X$. Since $\ang{g, g} = 1,$ and 
$$
\ang{T_{\G}g, g} = \int_X \bigg(\sum_{y \sim x} g(y) \bigg)\overline{g(x)}\dmu(x) = \int_X \deg(x)\dmu(x) = \deg_{\av}(\G),
$$
we immediately have that $m(T_\G) = \inf_{f\neq 0} \frac{\ang{T_{\G}f, f}}{\ang{f, f}} \leq \deg_{\av}(\G) \leq M(T_{\G}) = \sup_{f \neq 0} \frac{\ang{T_{\G}f, f}}{\ang{f, f}}$.

(2) The upper bound follows from Proposition~\ref{prop:boundedsa}. For the lower bound, consider the set $A = \{x\in X: \deg(x) = d\} \sub X$. By assumption, $\mu(A)>0$. Note that the graph $\G^4$ obtained from $\G$ by placing an edge between two distinct vertices if their distance in $\G$ is at most $4$ has bounded degree. Therefore, by \cite[Proposition~4.6]{kst1999}, $\G^4 \rest A$ admits a finite Borel coloring, and at least one of the monochromatic sets $A_0 \sub A$ obtained from this coloring has $\mu(A_0) > 0$. Then since $\G \rest A$ is $d$-regular and distinct points of $A_0$ are at distance greater than $4$ from each other, the inequalities in Equation~\ref{ineq: n(a)} when $A$ is replaced with $A_0$ become equalities, and we deduce that $\mu(\cN(A_0)) = d\mu(A_0)$. Now define the function $f\in L^2(X)$ as follows:
\begin{align*}
    f(x) = \begin{cases}
        \sqrt{d}, &\text{ if } x\in A_0,\\
        1, &\text{ if } x\in \cN(A_0),\\
        0, &\text{ otherwise}.
    \end{cases}
\end{align*}
Then we can compute 
$$
    \langle f,f\rangle = \int_X \abs{f(x)}^2\dmu(x) = d\mu(A_0) + \mu(\cN(A_0)) = 2d\mu(A_0),
$$
and, again using that distinct points of $A_0$ have distance greater than $4$,
\begin{align*}
    \langle T_\G f, f\rangle &= \int_X \sum_{y\sim x} f(y)f(x)\dmu(x)\\
    &= \int_{A_0} d\sqrt{d}\dmu(x) + \int_{\cN(A_0)} \Big(\sqrt{d} + \sum_{\substack{y\sim x,\\ y\notin A_0}} f(y)\Big)\dmu(x)\\
    &\geq 2d\sqrt{d}\mu(A_0).
\end{align*}
We conclude that 
$$
M(T_\G) \geq \frac{\langle T_\G f, f\rangle}{\langle f,f\rangle} \geq \frac{2d\sqrt{d}\mu(A_0)}{2d\mu(A_0)} = \sqrt{d},
$$
finishing the proof.
\end{proof}

Next, again analogous to the finite setting, we show that $m(T_\G)$ is always negative (unless the measure of the set of vertices incident on at least one edge is zero, in which case the adjacency operator is the zero operator). For finite graphs, this follows easily from the observation that the trace of the adjacency operator is 0. In our setting, the operators are defined on the infinite-dimensional Hilbert space $L^2(X)$; since there is no well-defined trace on all of $B(L^2(X))$, we need a different proof.

\begin{lem}\label{lem:mnegative}
    If $T_\G\neq 0$, then $m(T_\G) < 0$.
\end{lem}
\begin{proof}
As in the proof of Lemma~\ref{lem:avdeg}, we can find a positive-measure Borel set $A$ consisting of vertices with degree at least $1$ such that distinct vertices of $A$ have distance greater than $4$. Consider the function $f\in L^2(X)$ defined by 
\begin{align*}
    f(x) = \begin{cases}
        d, &\text{ if } x\in A,\\
        -1, &\text{ if } x\in \cN(A),\\
        0, &\text{ otherwise}.
    \end{cases}
\end{align*}
We can then easily compute
\begin{align*}
    \langle T_\G f, f\rangle &= \int_X \sum_{y\sim x} f(y) f(x)\dmu(x)\\
    &= \int_A -d \cdot \deg(x)\dmu(x) + \int_{\cN(A)} - \Big(d + \sum_{\substack{y\sim x,\\ y\notin A}} f(y)\Big)\dmu(x).
\end{align*}
The first integrand is clearly strictly negative, and since the maximum degree of $\G$ is equal to $d$, the second integrand is strictly negative as well. We thus conclude that 
$$
m(T_\G) \leq \frac{\langle T_\G f,f\rangle}{\langle f,f\rangle} < 0,
$$
finishing the proof.
\end{proof}

Finally, suppose $A \sub X$ is measurable and $\cH = \G \rest A$. In light of the discussion on induced subgraphs in Subsection~\ref{subsec:BorelGraphs},  
the adjacency operator $T_{\mathcal{H}}$ of $\mathcal{H}$ may be viewed as having domain equal to $L^2(A)$. By considering the direct sum decomposition $L^2(X) = L^2(A) \oplus L^2(X\setminus A)$, the following is immediate from Proposition~\ref{prop: block decomp T}(1):

\begin{cor}
\label{cor: block decomp M}
    Let $A \sub X$ be $\mu$-measurable, and let $\mathcal{H} = \G \rest A$. Then 
    $$
    m(T_{\G}) \leq m(T_{\cH}) \leq M(T_{\cH}) \leq M(T_{\G}).
    $$
\end{cor}

\subsection{The Laplacian operator}

The \emph{Laplacian operator} $L_{\G}\in L^2(X)$ is defined by
$$
(L_{\G}f)(x) = \deg(x)f(x) - \sum_{y \sim x} f(y) = (D_\G f)(x) - (T_{\G}f)(x),
$$
for all $f\in L^2(X)$ and $x\in X$, where $D_\G:X\to \C$ is the \textit{degree function} defined by $D_\G(x) = \deg(x)$ for $x\in X$. As in Example~\ref{exmp:adjacencymatrix}, we note that in the case of a finite graph, we recover the usual Laplacian matrix. We start with the following easy observation:

\begin{lem}
    $L_{\G}$ is bounded, self-adjoint, and positive.
\end{lem}
\begin{proof}
Since $\G$ has bounded degree, the degree function $D_\G:X\to \C$ is bounded and measurable, i.e., $D_\G\in L^\infty(X)$, and it takes values in $\R$. In particular, $D_\G$ defines a bounded self-adjoint operator on $L^2(X)$ by multiplication. Since $L_{\G}$ is a linear combination of $D_\G$ and $T_{\G}$, and $T_{\G}$ is bounded and self-adjoint by Proposition~\ref{prop:boundedsa}, we conclude that $L_{\G}$ is also bounded and self-adjoint.

Finally, we show that $L_{\G}$ is positive. We use the well-known fact that an operator $L$ is positive if and only if $\langle Lf,f\rangle\geq 0$ for every $f\in L^2(X)$. Before proceeding to $L_{\G}$, we first note that by Corollary~\ref{cor:transports}(1),
$$
    \langle D_\G f,f\rangle = \int_X \deg(x)\abs{f(x)}^2\dmu(x) = \int_X \sum_{y\sim x} \abs{f(y)}^2\dmu(x).
$$
Since we also trivially have $\deg(x)\abs{f(x)}^2 = \sum_{y\sim x} \abs{f(x)}^2$ for any $x\in X$, we can combine this with self-adjointness of $T_{\G}$ to get
\begin{align*}
    2\langle L_{\G} f,f\rangle &= \langle D_\G f,f\rangle + \langle D_\G f,f\rangle - \langle T_{\G}f,f\rangle- \langle f,T_{\G}f\rangle\\
    &= \int_X \sum_{y\sim x} \big[ \abs{f(x)}^2 + \abs{f(y)}^2 - f(y)\overline{f(x)} - f(x)\overline{f(y)} \big]\dmu(x)\\
    &= \int_X \sum_{y\sim x} \abs{f(x) - f(y)}^2 \dmu(x)\\
    &\geq 0.
\end{align*}
This finishes the proof.
\end{proof}

\begin{lem}\label{lem:LD-T}
    Suppose $\G$ is $d$-regular. Then $\sigma(L_{\G}) = d - \sigma(T_{\G}) = \{d-\lambda : \lambda\in \sigma(T_{\G})\}$. In particular, $M(L_\G) = d - m(T_\G) \leq 2d$ and $m(L_\G) = d - M(T_\G) = 0$.
\end{lem}
\begin{proof}
    First, we note that for any $\alpha\in \C$ and any bounded operator $T\in B(H)$ for some Hilbert space $H$,
    \begin{align*}
        \sigma(\alpha I - T) &= \{\lambda : (\alpha I - T) - \lambda I \text{ is not invertible}\}\\
        &= \{\lambda : T - (\alpha - \lambda) I \text{ is not invertible}\}\\
        &= \{\lambda  : \alpha - \lambda \in \sigma(T)\}\\
        &= \{\alpha - \eta : \eta\in \sigma(T)\}.
    \end{align*}
    The result now follows immediately from the fact that $L_{\G} = dI - T_{\G}$ when $\G$ is $d$-regular.
\end{proof}

Note that the fact that $m(L_\G)=0$ always holds. Indeed, $\sigma(L_\G)\sub [0,\infty)$ since $L_\G$ is positive, and it is easy to check that the constant functions are eigenfunctions with eigenvalue $0$.

\begin{remark}
    Note that the fact that $L_{\G}$ is positive also easily follows from Lemma~\ref{lem:LD-T} in case $\G$ is $d$-regular, as the lemma shows that its spectrum is contained in $d-\sigma(T_{\G})\sub [0,2d]$. In fact, the inclusion $\sigma(L_\G)\sub [0,2d]$, where $d$ is the maximum degree, remains valid for non-regular graphs, since $L_\G$ is positive and $\norm{L_\G} = \norm{D_\G - T_\G} \leq \norm{D_\G} + \norm{T_\G} \leq 2d$.
\end{remark}

\subsection{First basic properties}

We start with the observation that $\abs{m(T_{\G})}\leq M(T_{\G})$; hence, in particular, all spectral values are bounded in absolute value by $M(T_\G)$. This generalizes the classical Perron--Frobenius theory (see, for instance, \cite[Chapter~4]{spielman2025}).

\begin{lem}\label{lem:Mmax}
    We always have $\abs{m(T_{\G})}\leq M(T_{\G})$. In particular $\sigma(T_{\G})\sub [m(T_{\G}), M(T_{\G})] \sub [-M(T_{\G}), M(T_{\G})]$ and $\norm{T_\G} = M(T_\G)$.
\end{lem}
\begin{proof}
Fix $\eps>0$, write $T = T_{\G}$, and let $(f_n)_{n\in\N}$ be a sequence of approximate eigenfunctions of norm 1 in $L^2(X)$ for $m(T)$. By assumption, there exists $N\in \N$ such that for all $n\geq N$,
\begin{align*}
    \abs{m(T)} - \eps &\leq \abs{\langle Tf_n, f_n\rangle}\\
    &= \Bigg\vert \int_X \sum_{y\sim x} f_n(y)\overline{f_n(x)}\dmu(x) \Bigg\vert \\
    &\leq \int_X \sum_{y\sim x} \abs{f_n(y)}\abs{f_n(x)}\dmu(x)\\
    &= \langle T\abs{f_n},\abs{f_n}\rangle\\
    &\leq M(T).
\end{align*}
As $\eps>0$ was arbitrary, this finishes the proof.
\end{proof}

The first corollary is the analogue of the classical fact that the maximum eigenvalue has a positive eigenvector.

\begin{cor} 
\label{cor: Tf=Mf}
If $T_\G f = M(T_\G)f$ for some $f\in L^2(X)$, then $T_\G\abs{f} = M(T_\G)\abs{f}$. In particular, if $M(T_\G)$ is an eigenvalue, then it has a positive eigenfunction.
\end{cor}
\begin{proof}
We write $T=T_\G$ and $M=M(T_\G)$. Assume without loss of generality that $\norm{f} = 1$. A computation analogous to the one in Lemma~\ref{lem:Mmax} shows that $\langle T\abs{f},\abs{f}\rangle = M$. In particular,
$$
0 \leq \norm{T\abs{f} - M\abs{f}}^2 = \langle T\abs{f} - M\abs{f}, T\abs{f} - M\abs{f}\rangle = \norm{T\abs{f}}^2 - M^2 \leq 0.
$$
We conclude that, indeed, $T\abs{f} = M\abs{f}$.

\end{proof}

In fact, when $M(T_\G)$ is not an eigenvalue, we note that the same proof yields the fact that there always exists a sequence $(f_n)_{n\in\N}$ of positive approximate eigenfunctions for $M(T_\G)$. A similar argument gives the following:

\begin{cor}\label{cor:minusdinspectrum}
    Let $d$ be the maximum degree of $\G$. If $-d\in \sigma(T_{\G})$ and $(f_n)_{n \in \N}$ is a sequence of functions in $L^2(X)$ of norm $1$ satisfying $\norm{T_\G f_n + df_n}\to 0$, then $\norm{T_\G \abs{f_n} - d\abs{f_n}}\to 0$. In particular, $d \in \sigma(T_{\G})$.
\end{cor}

The next part of the classical Perron--Frobenius theory for finite graphs asserts that if $\G$ is connected, then 
the largest eigenvalue has multiplicity $1$ and its eigenvector has strictly positive entries. We observe that this goes through in our setting upon replacing connectedness by its natural analogue, ergodicity.

\begin{lem}\label{lem:strictlyposeigenvector}
    Assume $\G$ is ergodic and $M(T_\G)$ is an eigenvalue with a positive eigenfunction $f$. Then $f>0$.
\end{lem}
\begin{proof}
Let $X_0 = \{x\in X : f(x) = 0\}$ and assume $\mu(X_0)>0$. Since $f\geq 0$ and $M(T_\G)f(x) = T_\G f(x) = \sum_{y\sim x} f(y)$ for almost every $x\in X$, we see that for almost every $x\in X_0$ and $y\sim x$, $f(y) = 0$. In other words, $X_0 = B_1(X_0)$ modulo null sets, i.e., $X_0$ is $\G$-invariant. By ergodicity, we conclude that $X_0 = X$. This contradicts that $f\neq 0$.
\end{proof}

\begin{lem}\label{lem:ergeigenvalued}
    Assume $\G$ is ergodic and $M(T_\G)$ is an eigenvalue. Then $\ker(T_{\G} - M(T_\G)I) = \C f_0$ for some strictly positive function $f_0\in L^2(X)$.
\end{lem}
\begin{proof}
For convenience, we write $T=T_\G$ and $M=M(T_\G)$. We start with the following claim:

\begin{claim}
\label{claim: Tf=df2}
    If $Tf = Mf$ for some real-valued function $f\in L^2(X)$, then $f\geq 0$ or $f\leq 0$.
\end{claim}

\begin{proof}[Proof of Claim~\ref{claim: Tf=df2}]
Write $f = f_+ - f_-$ as its positive part minus its negative part. Note that $f_{\pm} = \frac{1}{2}(\abs{f} \pm f)$, and thus by Corollary~\ref{cor: Tf=Mf}, $Tf_{\pm} = Mf_{\pm}$. However, by construction, $f_+f_- = 0$; hence by Lemma~\ref{lem:strictlyposeigenvector}, we necessarily have $f_+ = 0$ or $f_- = 0$.
\end{proof}

We can now finish the proof. Using Corollary~\ref{cor: Tf=Mf}, let $f_0$ be a positive eigenfunction for $T$. From Lemma~\ref{lem:strictlyposeigenvector}, it follows that $f_0$ is strictly positive. Assume $Tf=Mf$ for some function $f$ which is not a multiple of $f_0$. By passing to the real or complex part of $f$ if necessary, we can assume $f$ is real-valued and not a multiple of $f_0$. In particular, there exists $r\in \R$ such that both $\mu(\{x\in X : f(x) - rf_0(x) > 0\}) > 0$ and $\mu(\{x\in X : f(x) - rf_0(x) < 0\}) > 0$. However, since $f - rf_0$ is an eigenfunction of $T$ with eigenvalue $M$, this contradicts Claim~\ref{claim: Tf=df2}, finishing the proof.
\end{proof}

\begin{remark}\label{rmk:nonstronglyergodic}
    Since approximate eigenfunctions are natural for self-adjoint operators on infinite-dimensional Hilbert spaces, one might expect that, if $\G$ is ergodic, then whenever $f_n\in L^2(X)$ are functions of norm 1 satisfying $\norm{T_\G f_n - M(T_\G )f_n}\to 0$ and $f_n\geq 0$, then $\mu(\{x\in X : f_n(x) = 0\})\to 0$. However, this can never be true in general. Indeed, consider any free ergodic but not strongly ergodic measure-preserving action of a countable discrete group $\Gamma$ on $(X,\mu)$, e.g., irrational rotation as an action of $\Z$ on $S^1$. The Schreier graph of this action yields a $d$-regular Borel pmp graph on $X$. Because the action is not strongly ergodic, we can find a sequence of Borel sets $A_n\sub X$ each having measure $\frac{1}{2}$ that are approximately invariant for the action. It is then easy to check that the indicator functions $1_{A_n}$ yield a sequence of approximate eigenfunctions for the (maximum) eigenvalue $d$. On the other hand, we will show below that the natural condition of spectral gap is in fact sufficient to deduce that any approximate eigenfunction of $d$ is approximately constant.
\end{remark}

Next, we record an easy generalization of the fact that eigenvectors of different eigenvalues are perpendicular to each other. We are mainly interested in the case where $\lambda_1 = d$ and the graph has spectral gap (see Subsection~\ref{subsec: spectral gap}).

\begin{lem}\label{lem:eigenvalueperp}
    Let $T\in B(H)$ be a self-adjoint operator. If $\lambda_1\neq \lambda_2\in \sigma(T)$, $\lambda_1$ is an eigenvalue, and $\norm{Tf_n - \lambda_2f_n}\to 0$, then for any eigenfunction $g$ of $\lambda_1$, we have $\langle f_n,g\rangle\to 0$.
\end{lem}
\begin{proof}
We have
\begin{align*}
    \lambda_1\langle f_n,g\rangle - \lambda_2\langle f_n,g\rangle &= \langle f_n,Tg\rangle - \lambda_2\langle f_n,g\rangle\\
    &= \langle Tf_n,g\rangle - \lambda_2\langle f_n,g\rangle\\
    &\to 0,
\end{align*}
from which the result follows immediately.
\end{proof}

\begin{cor}
    Assume $\G$ is $d$-regular and $\lambda\in\sigma(T_\G)$ with $\lambda\neq d$. If $\norm{T_\G f_n - \lambda f_n}\to 0$, then $\langle f_n,1\rangle \to 0$. 
\end{cor}

Let $L^2_0(X)$ denote the set of functions perpendicular to 1:
$$
L^2_0(X) = L^2(X) \ominus \C 1 = \{f \in L^2(X) : \int_X f(x)\dmu(x) = 0\}.
$$
Since we can perturb any sequence of functions that are approximately perpendicular to $1$ to a sequence of functions perpendicular to $1$, we obtain the following corollary:

\begin{cor}
    Assume $\G$ is $d$-regular. Then $\sigma(T|_{L^2_0(X)})$ is equal to either $\sigma(T)$ or $\sigma(T)\setminus\{d\}$.
\end{cor}

\subsection{Spectral gap}
\label{subsec: spectral gap}

To deal with the subtleties arising from approximately invariant sets (as outlined for instance in Remark~\ref{rmk:nonstronglyergodic}), we sometimes assume that the adjacency operator has spectral gap, a property that has appeared in many forms in the literature and has seen important applications across various disciplines. Just as for finite graphs, it is known that spectral gap connects to expansion properties of the graph; see for instance \cite{ln2011,HLS14}, as well as Section~\ref{sec:matchings}. We work with the following definition of spectral gap in the case when $\G$ is regular:

\begin{defn}
    Assume $\G$ is a $d$-regular Borel pmp graph. We say $\G$ has \textit{spectral gap} if there exists $\eps>0$ such that $\sigma(T_{\G})\sub [-d, d-\eps]\cup\{d\}$. We refer to the largest such $\eps$ as the \textit{spectral gap} of $\G$.
\end{defn}

\begin{remark}
    Since we restrict our attention to regular graphs, we may equivalently define the spectral gap using the Laplacian operator; in this case, the spectral gap is simply the largest $\eps$ such that $\sigma(L_\G)\sub \{0\}\cup [\eps,2d]$. Since $0$ is always an eigenvalue of $L_\G$ with the constant functions as eigenfunctions, this definition is sometimes more natural, although in our situation it is equivalent.
\end{remark}

There are many natural examples of (regular) Borel pmp graphs with spectral gap, notably Schreier graphs of group actions with spectral gap; this includes Bernoulli actions of non-amenable groups and actions of Property (T) groups. In contrast, it is well-known that Schreier graphs of actions of amenable groups never have spectral gap.

Before proceeding, we note that any finite connected graph has spectral gap. It is therefore perhaps unsurprising that certain classical techniques extend naturally to Borel graphs with spectral gap, while it is not always clear if they extend in general (see for instance Remark~\ref{rem: necessary?}).

We end this section with a few easy observations regarding Borel pmp graphs with spectral gap; we will see further uses in Sections~\ref{sec:bipartite} and~\ref{sec:matchings}. First, we note that it follows from Proposition~\ref{prop:isolated} that if $\G$ has spectral gap, then $\sigma(T|_{\ker(T-dI)^\perp}) = \sigma(T)\setminus\{d\}$. If moreover $\G$ is ergodic, then Lemma~\ref{lem:ergeigenvalued} implies the following:

\begin{cor}\label{cor:ergodicspectralgapspectrum}
    Let $\G$ be an ergodic $d$-regular Borel pmp graph with spectral gap. Then we have $\sigma(T|_{L^2_0(X)}) = \sigma(T)\setminus\{d\}$.
\end{cor}

In this case, using Proposition~\ref{prop:isolated}, it also follows easily that any sequence of approximate eigenfunctions for $d$ is approximately constant (see the proof of Theorem~\ref{thm:-dspecimpliesbipartite}).

Finally, another classical result for finite connected graphs is the fact that $\G$ is regular if the largest eigenvalue is equal to the maximum degree. We prove that this goes through when $M(T_\G)$ is a genuine eigenvalue, so for instance when $\G$ has spectral gap, although we do not know whether this assumption is necessary.

\begin{prop}
    Assume $\G$ is ergodic, the maximum degree of $\G$ is $d$, $M(T_\G) = d$, and $M(T_\G)$ is an eigenvalue. Then $\G$ is $d$-regular.
\end{prop}
\begin{proof}
Write $T=T_\G$ and let $f$ be a function of norm 1 satisfying $Tf=df$. By Corollary~\ref{cor: Tf=Mf}, we can assume that $f$ is positive, and then Lemma~\ref{lem:strictlyposeigenvector} implies that necessarily $f>0$. Furthermore, from Corollary~\ref{cor:ineq}, we deduce that
$$
d\int_X f(x)\dmu(x) = \int_X Tf(x)\dmu(x) = \int_X \sum_{y\sim x} f(y)\dmu(x) = \int_X \deg(x) f(x)\dmu(x).
$$
Assume for contradiction that $X_{<d} = \{x\in X : \deg(x) < d\}$ satisfies $\mu(X_{<d}) > 0$. Then the right-hand side of the above equation satisfies
$$
\int_X \deg(x) f(x)\dmu(x) \leq (d-1) \int_{X_{<d}} f(x) \dmu(x) + d \int_{X\setminus X_{<d}} f(x)\dmu(x) < d\int_X f(x)\dmu(x),
$$
yielding the desired contradiction.
\end{proof}

\section{Bipartiteness}\label{sec:bipartite}

Recall that a graph $\G$ on a set $X$ is \emph{bipartite} if $X$ can be partitioned into two disjoint $\G$-independent sets, or equivalently if $\G$ has no odd cycles. For a finite graph $\G$, it is well-known that if $\G$ is bipartite, then the spectrum of its adjacency operator is symmetric about 0, and that if $\G$ is moreover connected, then the converse holds as well. In fact, the negative of the maximum eigenvalue being in the spectrum is sufficient to deduce bipartiteness.

In this section, we show that this result goes through for bounded-degree Borel pmp graphs upon replacing bipartiteness with \textit{approximate measurable bipartiteness}. 

\begin{defn}
    Let $\G$ be a Borel graph on a standard probability space $(X,\mu)$. We say $\G$ is \textit{approximately measurably bipartite} if for every $\eps > 0$, there exist disjoint $\G$-independent measurable subsets $A,B\sub X$ such that $\mu(A\cup B) > 1-\eps$.
\end{defn}

\begin{remark}
    Note that a Borel graph $\G$ is approximately measurably bipartite if and only if $\chi^\ap_\mu(\G) \leq 2$. This is consistent with the well-known classical fact that bipartiteness and 2-colorability are equivalent.
\end{remark}

For clarity, we make the following precise:

\begin{defn}
    For a bounded, self-adjoint operator $T$, we say that the spectrum of $T$ is \textit{symmetric} if, whenever $\lambda \in \sigma(T)$, also $-\lambda \in \sigma(T)$.
\end{defn}

We first show that the adjacency operator of any approximately measurably bipartite graph has symmetric spectrum. This argument requires no assumptions on $\G$ beyond that it is a bounded-degree Borel pmp graph. For the converse (see Theorem~\ref{thm:-dspecimpliesbipartite}), we additionally assume that $\G$ is ergodic and has spectral gap (although see Remark~\ref{rem: necessary?}).

\begin{thm}\label{thm:bipartiteimpliessymspec}
    Let $(X, \mu)$ be a standard probability space, and let $\G$ be a bounded-degree Borel pmp graph on $X$. If $\G$ is approximately measurably bipartite, then the spectrum of $T_{\G}$ is symmetric. 
\end{thm}
\begin{proof}
We write $T = T_{\G}$ for convenience, and we denote by $d$ the maximum degree of $\G$. Assume $\G$ is approximately measurably bipartite. Let $\lambda \in \sigma(T)$, and fix $1 > \eps > 0$. It is enough to show that there is $g \in L^2(X)$ such that
$$
\norm{Tg + \lambda g}^2 < \eps.
$$
Let $\eps_0 = \frac{\eps}{1+2d^2+2\lambda^2}$. Since $\lambda \in \sigma(T)$, there is a unit vector $f \in L^2(X)$ such that $\norm{Tf - \lambda f}^2 < \eps_0$. By Lemma~\ref{lem: dct}, there is $\delta > 0$ such that for any $\mu$-measurable $B$,
\begin{equation}\label{eq:fbdd}
\mu(B) < \delta \implies \int_B \abs{f(x)}^2 \dmu(x) < \eps_0.
\end{equation}
Since $\G$ is approximately measurably bipartite, there are disjoint measurable $\G$-independent sets $A_1, A_2 \sub X$ such that $\mu(A_1 \cup A_2) > 1 - \frac{\delta}{d(d+1)}$. Let $A_3 = X \setminus (A_1 \cup A_2)$, so that $\mu(A_3) < \frac{\delta}{d(d+1)}$.

By construction, there are $f_1 \in L^2(A_1), f_2 \in L^2(A_2), f_3 \in L^2(A_3)$ such that $f = f_1 + f_2 + f_3$. Note that $\norm{f_1}, \norm{f_2}, \norm{f_3} \leq 1$. Define $g = f_1 - f_2 + f_3$. In particular, $\abs{g} = \abs{f}$, $g=f$ on $A_1\cup A_3$, and $g=-f$ on $A_2$. We claim that $g$ is our desired function.

First, define $B = A_1\cap \cN(A_3)$. By Corollary~\ref{cor:neighbors}, we have $\mu(B)\leq d\mu(A_3)$. Furthermore, for every $x\in A_1\setminus B$, we have by construction that all its neighbors are in $A_2$; hence
$$
Tg(x) = \sum_{y\sim x} g(y) = \sum_{y\sim x} -f_2(y) = - \sum_{y\sim x} f(y) = - Tf(x),
$$
and in particular
$$
Tg(x) + \lambda g(x) = -Tf(x) + \lambda f(x).
$$
Similarly, for every $x\in A_2$, we get
$$
Tg(x) = \sum_{y\sim x,y\in A_1} g(y) + \sum_{y\sim x,y\in A_3} g(y) = \sum_{y\sim x,y\in A_1} f_1(y) + \sum_{y\sim x,y\in A_3} f_3(y) = \sum_{y\sim x} f(y) = Tf(x)
$$
and thus
$$
Tg(x) + \lambda g(x) = Tf(x) - \lambda f(x).
$$

Combining the above, we can now compute
\begin{align}\label{eq:Tglg1}
\begin{split}
    \norm{Tg+\lambda g}^2 &= \int_X \abs{Tg(x) + \lambda g(x)}^2\dmu(x)\\
    &= \int_{(A_1\setminus B)\cup A_2} \abs{Tg(x) + \lambda g(x)}^2\dmu(x) + \int_{B\cup A_3} \abs{Tg(x) + \lambda g(x)}^2\dmu(x)\\
    &= \int_{(A_1\setminus B)\cup A_2} \abs{Tf(x) - \lambda f(x)}^2\dmu(x) + \int_{B\cup A_3} \abs{Tg(x) + \lambda g(x)}^2\dmu(x).
\end{split}
\end{align}
By assumption, the first summand is bounded by $\eps_0$. For the second summand, we first note that the triangle inequality implies that $\abs{a+b}^2 \leq 2\abs{a}^2 + 2\abs{b}^2$, and thus similarly to \eqref{eq:Tbdd}, we get
\begin{align}\label{eq:Tglg2}
\begin{split}
    \int_{B\cup A_3} \abs{Tg(x) + \lambda g(x)}^2\dmu(x) &\leq 2\int_{B\cup A_3} \abs{Tg(x)}^2\dmu(x) + 2 \int_{B\cup A_3} \abs{\lambda g(x)}^2\dmu(x)\\
    &\leq 2d\int_{B\cup A_3} \sum_{y\sim x}\abs{g(y)}^2\dmu(x) + 2\lambda^2 \int_{B\cup A_3} \abs{g(x)}^2\dmu(x)\\
    &= 2d\int_{B\cup A_3} \sum_{y\sim x}\abs{f(y)}^2\dmu(x) + 2\lambda^2 \int_{B\cup A_3} \abs{f(x)}^2\dmu(x).
\end{split}
\end{align}
Since $\mu(B\cup A_3)\leq (d+1)\mu(A_3) <\delta$, line~\eqref{eq:fbdd} tells us that the second summand is bounded above by $2\lambda^2\eps_0$. In order to deal with the first summand, we observe that Corollary~\ref{cor:neighbors} implies that $\mu(\cN(B\cup A_3))\leq d\mu(B\cup A_3) \leq d(d+1)\mu(A_3) < \delta$, and hence using Corollary~\ref{cor:ineq} and line~\eqref{eq:fbdd}, we get
\begin{align}\label{eq:Tglg3}
\begin{split}
    \int_{B\cup A_3} \sum_{y\sim x}\abs{f(y)}^2\dmu(x) &\leq \int_{\cN(B\cup A_3)} \deg(x)\abs{f(x)}^2\dmu(x)\\
    &\leq d \int_{\cN(B\cup A_3)} \abs{f(x)}^2\dmu(x)\\
    &< d\eps_0.
\end{split}
\end{align}
Combining \eqref{eq:Tglg1}, \eqref{eq:Tglg2}, and \eqref{eq:Tglg3}, we get
$$
\norm{Tg+\lambda g}^2 < \eps_0 + 2d^2\eps_0 + 2\lambda^2 \eps_0 = \eps.
$$
This finishes the proof of the theorem.
\end{proof}

The next theorem shows in particular that the converse holds if $\G$ is ergodic and $d$-regular and has spectral gap.

\begin{thm}\label{thm:-dspecimpliesbipartite}
    Let $(X, \mu)$ be a standard probability space, and let $\G$ be an ergodic $d$-regular Borel pmp graph on $X$ with spectral gap. If $-d\in \sigma(T_{\G})$, then $\G$ is approximately measurably bipartite. 
\end{thm}
\begin{proof} 
Assume that $-d\in \sigma(T_{\G})$ and fix $1>\eps>0$. We show that there exist disjoint $\G$-independent Borel subsets $A,B\sub X$ such that $\mu(A\cup B) > 1-\eps$. Let $(f_n)_{n \in \N}$ be a sequence of functions in $L^2_0(X)$ of norm 1 such that
$$
\norm{Tf_n + df_n} \to 0.
$$
Note that we can assume that all the $f_n$ are real-valued. Indeed, if $f_n = f_n^{(1)} + i f_n^{(2)}$ is the decomposition of $f_n$ into its real and imaginary parts, then by definition of $T$, the decomposition of $Tf_n$ into its real and imaginary parts is given by $Tf_n = Tf_n^{(1)} + i Tf_n^{(2)}$.
\begin{claim} 
\label{claim: constant}
Without loss of generality, we can assume $\abs{f_n}$ is a constant function for each $n \in \N$.
\end{claim}

\begin{proof}[Proof of Claim~\ref{claim: constant}] By Corollary~\ref{cor:minusdinspectrum}, we have $\norm{T\abs{f_n} - d\abs{f_n}}\to 0$, i.e., $(g_n)_{n \in \N} = (\abs{f_n})_{n \in \N}$ is a sequence of approximate eigenfunctions for $d$. For each $n \in \N$, we decompose $g_n = g_n^{(0)} + g_n^{(1)}$ where $g_n^{(0)} = \int_X g_n(x)\dmu(x) \in \C1$ and $g_n^{(1)}\in L^2_0(X)$. Combining the assumptions with Lemma~\ref{lem:ergeigenvalued} and Proposition~\ref{prop:isolated}, we obtain $c>0$ such that 
$$
\norm{Tg_n^{(1)} - dg_n^{(1)}} \geq c \norm{g_n^{(1)}}
$$
for all $n$. However, since $Tg_n^{(0)} = dg_n^{(0)}$, we also have $\norm{Tg_n^{(1)} - dg_n^{(1)}} = \norm{Tg_n - dg_n}\to 0$, and so we conclude that $\norm{g_n^{(1)}}\to 0$. 

Consider now the function $\tilde f_n$ defined by
$$
\tilde f_n(x) = f_n(x) - \mathrm{sgn}(f_n(x)) g_n^{(1)}(x).
$$
Then we have
\begin{align*}
    \abs{\tilde f_n(x)} = \begin{cases}
        \abs{f_n(x) - g_n^{(1)}(x)} = \abs{\,\abs{f_n(x)} - g_n^{(1)}(x)} = \abs{g_n^{(0)}(x)}, &\text{ if } f_n(x)\geq 0, \text{ and}\\
        \abs{f_n(x) + g_n^{(1)}(x)} = \abs{ - \abs{f_n(x)} + g_n^{(1)}(x)} = \abs{g_n^{(0)}(x)}, &\text{ if } f_n(x) < 0.
    \end{cases}
\end{align*}
In particular, $\abs{\tilde f_n}$ is constant. Moreover, 
\begin{align*}
    \norm{T\tilde f_n + d\tilde f_n} &\leq \norm{Tf_n + df_n} + \norm{Tg_n^{(1)}} + d\norm{g_n^{(1)}}\\
    &\leq \norm{Tf_n + df_n} + 2d\norm{g_n^{(1)}}\\
    &\to 0.
\end{align*}
This finishes the proof of the claim.
\end{proof}

Using the claim (and normalizing if necessary), we can thus find a sequence $(f_n)_{n \in \N}$ of functions in $L^2_0(X)$ such that $\abs{f_n} = 1$ for each $n \in \N$ and $\norm{Tf_n + df_n} \to 0$. Fix $N\in \N$ such that $\norm{Tf_N + df_N}\leq \eps$. Then the set $Y = \{x\in X : \abs{Tf_N(x) + df_N(x)} \geq \sqrt{\eps} \}$ necessarily has $\mu(Y)\leq \eps$. 

Without loss of generality (by passing to the real part of $f_N$ if necessary), we have $f_N(x)\in \{\pm 1\}$ for every $x\in X$. Fix $x\in X\setminus Y$ and assume $f_N(x) = 1$. We claim that every $y\in X$ adjacent to $x$ satisfies $f_N(y) = -1$. Indeed, since $x\notin Y$, we have 
$$
\abs{Tf_N(x) + df_N(x)} = \sum_{y\sim x} f_N(y) + d \leq \sqrt{\eps}.
$$
If there exists $y_0\sim x$ with $f_N(y_0) = 1$, then 
$$
\sum_{y\sim x} f_N(y) + d = \sum_{\substack{y\sim x\\ y\neq y_0}} f_N(y) + 1 + d \geq 2 > \sqrt{\eps}.
$$
Similarly, if $x\in X\setminus Y$ satisfies $f_N(x) = -1$, then every $y\in X$ adjacent to $x$ satisfies $f_N(y) = 1$. In particular, if we define $A = \{x\in X\setminus Y : f_N(x) = 1\}$ and $B = \{x\in X\setminus Y : f_N(x) = -1\}$, then both $A$ and $B$ are $\G$-independent, and $\mu(A\cup B) = \mu(X\setminus Y) \geq 1-\eps$. This finishes the proof.
\end{proof}

We note that Theorems~\ref{thm:bipartiteimpliessymspec} and \ref{thm:-dspecimpliesbipartite} yield Theorem~\ref{thm:bipartite_intro} from the introduction. The following is also immediate:

\begin{cor}\label{cor:bipartite}
    Let $(X, \mu)$ be a standard probability space, and let $\G$ be an ergodic $d$-regular Borel pmp graph on $X$ with spectral gap. Then the following are equivalent:
    \begin{enumerate}
        \item $\chi_\mu^\ap(\G) \leq 2$, i.e., $\G$ is approximately measurably bipartite.
        \item $\sigma(T_{\G})$ is symmetric.
        \item $-d\in \sigma(T_\G)$.
    \end{enumerate}
\end{cor}

\begin{exmp}\label{ex:irrationalrotation2}
    Let $\alpha$ be an irrational number, and consider again the $2$-regular Borel pmp graph $\G$ on $X = [0, 1)$ such that $x \sim y$ if and only if $x + \alpha \equiv_1 y$ or $y + \alpha \equiv_1 x$. 
    
    It is well-known that the approximate $\mu$-measurable chromatic number of $\G$ is $2$; to see this, let $\eps > 0$, and let $C = [0, \gamma)$, where $\gamma < \min\{\alpha, 1 - \alpha, \eps \}$ is positive. Put $A = X \setminus C$, and define a measurable $2$-coloring on $\G \rest A$ by $c(x) = 0$ if the minimum number of rotations of $x$ by $\alpha$ to get into $C$ is even, $c(x) = 1$ otherwise.

    From the results in this section, it now follows immediately that $\sigma(T_{\G})$ is symmetric. In particular, $m(T_{\G}) = -M(T_{\G}) = -2\in \sigma(T_{\G})$.
\end{exmp}

\begin{remark}
\label{rem: necessary?}
    We note that ergodicity is a necessary assumption for the implication $(2)\Rightarrow (1)$. Indeed, it is easy to construct finite non-ergodic (i.e., disconnected) graphs with symmetric spectrum that are not $2$-colorable. For instance, one can take the graph on $9$ vertices that is the disjoint union of a triangle, a square, and a single edge connecting two vertices. On the other hand, we do not know whether the $(2)\Rightarrow (1)$ implication in Corollary~\ref{cor:bipartite} remains true without the assumptions of regularity and/or spectral gap. Indeed, for regularity, we simply observe that there exist finite non-regular bipartite graphs. As for spectral gap, we note that the graph $\G$ arising from irrational rotation in Example~\ref{ex:irrationalrotation2} satisfies all conditions in Corollary~\ref{cor:bipartite}, but does not have spectral gap. Indeed, since $\Z$ is amenable, $\Z\act S^1$ is not strongly ergodic. In particular, there exists a sequence of almost invariant sets, and hence a sequence $(f_n)_{n \in \N}$ of $\{0,1\}$-valued approximate eigenfunctions for $2\in \sigma(T_\G)$ with $\mu(\{x\in X : f_n(x) = 1\}) = \frac{1}{2}$ for each $n \in \N$. Compare also with \cite[Theorem~3.8]{ck2013}, which implies that whenever $\G$ is hyperfinite, $\chi_\mu^\ap(\G) = \chi^*(\G)$, where $\chi^*(G) = \min\{\chi(\G \rest A)  : A\sub X \text{ Borel}, \mu(A) = 1\}$. 
\end{remark}

\section{Upper bounds on the chromatic number}\label{sec:upperbounds}

In this section, we give novel upper bounds on the approximate measurable chromatic number of a bounded-degree Borel pmp graph. Before we proceed to the main results, we prove two lemmas. First, recall that, if $Y$ is a Polish space, then the space $[Y]^{< \infty}$ of finite subsets of $Y$ is a standard Borel space. A Borel map $L : X \to [Y]^{< \infty}$ is a \emph{Borel ($Y$-)list assignment}. A \emph{Borel $L$-coloring} of $\G$ is a Borel proper $Y$-coloring $c$ of $\G$ with the property that, for each $x \in X$, $c(x) \in L(x)$. The proof of the following proposition requires just a small modification to the proof of \cite[Proposition~4.6]{kst1999}; see also \cite[Proposition~3.8]{cmt2016}:

\begin{prop}
\label{prop: deg plus one list coloring}
    Let $\G$ be a locally finite Borel graph on a standard Borel space $X$. Let $Y$ be a Polish space, and suppose $L : X \to [Y]^{< \infty}$ is a Borel list assignment such that $\deg(x) < \abs{L(x)}$ for each $x \in X$. Then $\G$ admits a Borel $L$-coloring.
\end{prop}

\begin{proof}
    Since $\G$ is locally finite, there are $\G$-independent Borel sets $A_0, A_1, \dots$ such that $X = \bigsqcup_{n \in \N} A_n$ by \cite[Proposition~4.5]{kst1999}. Let $<_Y$ be a Borel linear ordering on $Y$, and for each $x \in A_0$, define $c(x)$ to be the $<_Y$-least element of $L(x)$. 
    
    Now fix $n \in \N$, and assume that the partial Borel $L$-coloring $c$ has been extended to the domain $A_0 \cup \cdots \cup A_n$. For each $x \in A_{n + 1}$, let 
    $$L_{n + 1}(x) = L(x) \setminus \{c(y) : y \in \cN(x) \cap (A_0 \cup \cdots \cup A_n) \}.$$ 
    Then $\abs{L_{n + 1}(x)} \geq 1$ for each $x \in A_{n + 1}$; indeed, if $x \in A_{n + 1}$, then 
    $$\abs{\{c(y) : y \in \cN(x) \cap (A_0 \cup \cdots \cup A_n) \}} \leq \deg(x),$$
    and so
    \begin{align*}
        \abs{L_{n + 1}(x)} & = \abs{L(x)} - \abs{\{c(y) : y \in \cN(x) \cap (A_0 \cup \cdots \cup A_n) \}}  \\
        & \geq \deg(x) + 1 - \deg(x) \\
        & = 1.
    \end{align*}
    So extend $c$ to $A_{n + 1}$ by defining $c(x)$ to be the $<_Y$-least element of $L_{n + 1}(x)$ for each $x \in A_{n + 1}$. The map $c$ that results from this inductive procedure is a Borel $L$-coloring of $\G$.
\end{proof}

We also prove the following technical lemma, which in fact can be strengthened (see Remark~\ref{rem: anton}):

\begin{lem}
\label{lem: list coloring from r}
    Let $\G$ be a locally finite Borel graph on a standard probability space $(X, \mu)$. Let $M$ be a positive integer, let $r \in (0, 1)$, and let $(A_n)_{n \in \N}$ be a sequence of pairwise disjoint Borel subsets of $X$. For each $n \in \N$, define $B_n = \bigcup_{k \leq n} A_k$, $\G_n = \G \rest (X \setminus B_{n - 1})$, and set $A_{-1} = B_{-1} = \emptyset$. Assume the following conditions hold for each $n \in \N$:
    \begin{enumerate}
        \item $\mu(X \setminus B_n) \leq r\mu(X \setminus B_{n -1})$, and
        \item $\deg_{\G_n}(x) \leq M$ for all $x \in A_n$.
    \end{enumerate}
    Then $\chi_{\mu}^{\ap}(\G) \leq M + 1$.
\end{lem}

\begin{proof}
    Let $\eps > 0$. By (1), for each $n \in \N$,
    $$\mu(X \setminus B_n) \leq r^{n + 1}.$$
    Since $r \in (0, 1)$, there is $N \in \N$ such that, for all $n \geq N$,
    $$\mu(X \setminus B_n) < \eps,$$
    so that $\mu(B_n) > 1 - \eps$.

    For each $x \in X$, let $L_N(x) = \{0, \dots, M\}$. Note that, for each $x \in A_N$, $\deg_{\G_N}(x) \leq M$, and so by Proposition~\ref{prop: deg plus one list coloring}, there is a Borel $L_N$-coloring $c_N$ of $\G \rest A_N$.

    Now fix $k \in \N$ with $k < N$ and assume that a Borel $(M + 1)$-coloring $c_{N - k}$ has been defined on $\G \rest (A_{N - k} \cup \cdots \cup A_N)$. Now, for each $x \in A_{N - (k + 1)}$, define 
    $$L_{N - (k + 1)}(x) = L_N(x) \setminus \{c_{N - k}(y) : y \in \cN(x) \cap (A_{N - k} \cup \cdots \cup A_N) \};$$
    then by (2),
    \begin{align*}
        \abs{L_{N - (k + 1)}(x)} & \geq \abs{L_N(x)} - \abs{\cN(x) \cap (A_{N - k} \cup \cdots \cup A_N)} \\
        & = (M + 1) - \abs{\cN(x) \cap (A_{N - k} \cup \cdots \cup A_N)} \\
        & \geq (\deg_{\G_{N - (k + 1)}}(x) + 1) - \abs{\cN(x) \cap (A_{N - k} \cup \cdots \cup A_N)} \\
        & = (\deg_{\G_{N - (k + 1)}}(x) - \abs{\cN(x) \cap (A_{N - k} \cup \cdots \cup A_N)}) + 1 \\
        & \geq \abs{\cN(x) \cap A_{N - (k + 1)}} + 1 \\
        & = \deg_{\G \rest (A_{N - (k + 1)})}(x) + 1.
    \end{align*}
    There is then a Borel $L_{N - (k + 1)}$-coloring $\tilde c_{N - (k + 1)}$ of $\G \rest (A_{N - (k + 1)})$ by Proposition~\ref{prop: deg plus one list coloring}. 
    
    Note that $c_{N - (k + 1)} = \tilde c_{N - (k + 1)}\cup c_{N - k}$ is an $(M + 1)$-coloring of $\G \rest (A_{N - (k + 1)} \cup \cdots \cup A_N)$; indeed, by pairwise disjointness of the sequence $(A_n)_{n \in \N}$, $\tilde c_{N - (k + 1)} \cup c_{N - k}$ is well-defined. Let $x, y \in A_{N - (k + 1)} \cup A_{N - k} \cup \cdots \cup A_N$ be $\G$-adjacent. If $x, y \in A_{N - (k + 1)}$, then because $\tilde c_{N - (k + 1)}$ is a coloring, $\tilde c_{N - (k + 1)}(x) \neq \tilde c_{N - (k + 1)}(y)$. Similarly, if $x, y \in A_{N - k} \cup \cdots \cup A_N$, then $c_{N - k}(x) \neq c_{N - k}(y)$. If $x \in A_{N - (k + 1)}$ and $y \in A_{N - k} \cup \cdots \cup A_N$, then since $c_{N - k}(y) \notin L_{N - (k + 1)}(x)$, $\tilde c_{N - (k + 1)}(x) \neq c_{N - k}(y)$.

    Therefore, the map $c = c_0 = \tilde c_0 \cup \cdots \cup \tilde c_N$ is a Borel $(M + 1)$-coloring of $\G \rest B_N$.
\end{proof}

\begin{remark}
\label{rem: anton}
    In personal communication, Anton Bernshteyn pointed out that Lemma~\ref{lem: list coloring from r} can be strengthened as follows. Let $\G$ be a locally finite Borel graph on a standard probability space $(X, \mu)$, and let $M$ be a positive integer. Assume that, for each Borel set $A \sub X$ such that $\mu(A) > 0$, the minimum degree of the induced subgraph $\G \rest A$ is at most $M$. Then $\chi_{\mu}^{\ap}(\G) \leq M + 1$. 
    
    The proof is similar: First, recursively construct sequences $(\G_n)_{n \in \N}$ of subgraphs of $\G$ and $(A_n)_{n \in \N}$ of subsets of $X$ by setting $\G_0 = \G, A_0 = \{x \in X : \deg_{\G_0}(x) \leq M \}$, and, having defined $\G_n, A_n$, setting
    \begin{align*}
        \G_{n + 1} & = \G \rest (X \setminus (A_0 \cup \cdots \cup A_n)), \\
        A_{n + 1} & = \{x \in V(\G_{n + 1}) : \deg_{\G_{n + 1}}(x) \leq M \}.
    \end{align*}
    Write $A_{\infty} = \bigcup_{n \in \N} A_n$ and $\cH = \G \rest (X \setminus A_{\infty})$. We claim that, for each $x \in X \setminus A_{\infty}$, we have $\deg_{\cH}(x) > M$; otherwise, there is $x \in X \setminus A_{\infty}$ with $\cH$-degree at most $M$. Since $\G$ is locally finite, we have $\deg_{\G}(x) = k$ for some nonnegative integer $k$. Then since $\deg_{\cH}(x) \leq M$, there is $\ell \geq k - M$ such that $\ell$ neighbors $y_0, \dots, y_{\ell - 1}$ of $x$ belong to $A_{\infty}$. For each $i < \ell$, there is $n_i$ such that $y_i \in A_{n_i}$. But then if $N = \max\{n_i : i < \ell \}$, then $\deg_{\G_{N + 1}}(x) \leq M$, so that $x \in A_{N + 1}$, a contradiction.
    So $\mu(A_{\infty}) = 1$; otherwise, $\G \rest (X \setminus A_{\infty})$ has minimum degree greater than $M$, which contradicts the initial assumption on $\G$. It follows that, for any $\eps > 0$, there is $N \in \N$ such that $\mu(A_0 \cup \cdots \cup A_N) > 1 - \eps$, and then the backwards list-coloring argument from the proof of Lemma~\ref{lem: list coloring from r} can be repeated.
\end{remark}

\subsection{Wilf's theorem}

A classical theorem of Wilf gives a simple upper bound on the chromatic number of a finite graph in terms of the largest eigenvalue of its adjacency matrix.

\begin{thm}
    \cite{wilf1967} Let $\G$ be a finite graph. Then
    $$\chi(\G) \leq \floor{M(T_{\G})} + 1.$$
\end{thm}

We show that this theorem carries over to the pmp setting when the chromatic number is replaced by the approximate measurable chromatic number; this establishes Theorem~\ref{thm:wilf_intro} from the introduction.

\begin{thm}
\label{thm: wilf pmp}
    Let $\G$ be a bounded-degree Borel pmp graph on a standard probability space $(X, \mu)$. Then
    $$\chi_{\mu}^{\ap}(\G) \leq \floor{M(T_{\G})} + 1.$$
\end{thm}

\begin{proof}
    Let $\eps > 0$, and write $\lambda = \floor{M(T_{\G})}$. Put
    $$r = \frac{\lambda + s}{\lambda + 1},$$
    where $s \in (0, 1)$ is such that $\lambda + s > M(T_{\G})$. Let $\G_0 = \G$, $A_0 = \{x \in X : \deg_{\G_0}(x) \leq \lambda \}$, and $A_{-1} = B_{-1} = \emptyset$. 
    We recursively define $B_n = \bigcup_{k \leq n} A_k$, $\G_{n+1} = \G \rest (X \setminus B_{n})$, and $A_{n + 1} = \{x \in X \setminus B_n : \deg_{\G_{n + 1}}(x) \leq \lambda \}$.
    We claim that the following conditions hold:
    \begin{enumerate}
        \item $A_n \cap A_m = \emptyset$ for all $m \neq n$,
        \item $\mu(X \setminus B_n) \leq r\mu(X \setminus B_{n - 1})$, and
        \item $\deg_{\G_n}(x) \leq \lambda$ for all $x \in A_n$.
    \end{enumerate}
    Conditions (1) and (3) are immediate from the construction, so we only need to prove (2).

    \begin{claim}
    \label{claim: ineq1}
        $\mu(X \setminus B_{n + 1}) \leq r\mu(X \setminus B_n)$ for all $n\in \N$.
    \end{claim}

    \begin{proof}[Proof of Claim~\ref{claim: ineq1}.]
        Suppose not. Then there exists $n\in \N$ such that $\mu(X \setminus B_{n + 1}) > r\mu(X \setminus B_n)$. So 
        \begin{align*}
            \deg_{\av}(\G_{n + 1}) & = \frac{1}{\mu(X \setminus B_n)} \int_{X \setminus B_n} \deg_{\G_{n + 1}}(x)\dmu(x) \\
            & \geq \frac{1}{\mu(X \setminus B_n)} \int_{X \setminus B_{n + 1}} \deg_{\G_{n + 1}}(x)\dmu(x) \\
            & \geq \frac{1}{\mu(X \setminus B_n)}\mu(X \setminus B_{n + 1})(\lambda + 1) \\
            & > \frac{1}{\mu(X \setminus B_n)}\Big(\mu(X \setminus B_n)\frac{\lambda + s}{\lambda + 1}\Big)(\lambda + 1) \\
            & = \lambda + s > M(T_{\G}),
        \end{align*}
        contradicting the facts from Lemma~\ref{lem:avdeg} and Corollary~\ref{cor: block decomp M} that $\deg_{\av}(\G_{n + 1}) \leq M(T_{\G_{n + 1}}) \leq M(T_{\G})$. This proves the claim.
    \end{proof}

    From the properties (1)--(3), we deduce by Lemma~\ref{lem: list coloring from r} that $\chi_{\mu}^{\ap}(\G) \leq \lambda + 1 = \floor{M(T_{\G})} + 1$, as desired.
 \end{proof}
 
We note that by Corollary \ref{cor:bireg} this bound can improve on the bound from Brooks's theorem quadratically (although that particular case of biregular graphs is not really of any interest as these are trivially 2-colorable). We also point out that Lemma~\ref{lem:avdeg}(2) implies that quadratic improvement is the best possible.

\subsection{Graphs generated by functions}

Let $X$ be a set, and let $f_1, \dots, f_n : X \to X$ be functions. The \emph{graph generated by $f_1, \dots, f_n$}, denoted $\G_{f_1, \dots, f_n}$, is the graph on $X$ such that $x, y \in X$ are adjacent if $x \neq y$ and there is some $i \leq n$ such that either $f_i(x) = y$ or $f_i(y) = x$. A well-known classical greedy algorithm argument demonstrates that $\chi(\G_{f_1, \dots, f_n}) \leq 2n + 1$; this is best possible (see Remark~\ref{rem: paley}). 

The \emph{directed graph generated by $f_1, \dots, f_n$}, denoted $\overrightarrow{\G}_{f_1, \dots, f_n}$, is the directed graph on $X$ such that there is a directed edge from $x$ to $y$ if $x \neq y$ and $f_i(x) = y$ for some $i \leq n$. For each $x \in X$, the \emph{out-degree} of $x$, denoted $\deg^+(x)$, is the number of out-edges incident on $x$, and the \emph{in-degree} of $x$, denoted $\deg^-(x)$, is the number of in-edges incident on $x$.

Next, using properties of the in-degree and out-degree in the directed graph $\overrightarrow{\G}_{f_1, \dots, f_n}$, we prove that $2n + 1$ is an upper bound on the approximate measurable chromatics number for the underlying undirected graph $\G_{f_1, \dots, f_n}$ in the case when it is pmp, thereby establishing Theorem~\ref{thm:functions_intro} from the introduction. The second part of the proof is similar to the proof of Theorem~\ref{thm: wilf pmp}, but we provide the details for completeness.

\begin{thm}\label{thm:nfunctions}
    Let $(X, \mu)$ be a standard probability space, and let $f_1, \dots, f_n : X \to X$ be bounded-to-one Borel functions such that the graph $\G_{f_1, \dots, f_n}$ is pmp. Then $\chi_{\mu}^{\ap}(\G_{f_1, \dots, f_n}) \leq 2n + 1$.
\end{thm}

\begin{proof}
    For convenience, we write $\G = \G_{f_1, \dots, f_n}$ and $\cD = \overrightarrow{\G}_{f_1, \dots, f_n}$. Define a Borel map $\varphi : E_{\G} \to [0, \infty)$ by $\varphi(x, y) = 1$ if $x \neq y$ and there is $i \leq n$ such that $f_i(x) = y$ and $\varphi(x, y) = 0$ otherwise. For each $x \in X$, we have $\out\,\varphi(x) = \deg^+(x)$ and $\ins\,\varphi(x) = \deg^-(x)$. Therefore, by Proposition~\ref{prop:Anush}, 
    $$
        \deg^+_{\av}(\cD) = \int_X \deg^+(x) \dmu(x) = \int_X \deg^-(x) \dmu(x) = \deg_{\av}^-(\cD).
    $$
    Note that $\deg^+(x) \leq n$ for all $x \in X$, so that $\deg_{\av}^+(\cD) \leq n$. Therefore also $\deg_{\av}^-(\cD) \leq n$. Note further that, if $A \sub X$ is Borel, then $\deg_{\av}^+(\cD \rest A) = \deg_{\av}^-(\cD \rest A)$, as witnessed by the Borel transport $\varphi_A : E_{\G \rest A} \to [0, \infty)$ given by $\varphi_A(x, y) = 1$ if $x, y \in A$, $x \neq y$, and there is $i \leq n$ such that $f_i(x) = y$ and $\varphi_A(x, y) = 0$ otherwise; furthermore, $\deg_{\av}^+(\cD \rest A) \leq n$, and so $\deg_{\av}^-(\cD \rest A) \leq n$.

    Let now $\cG_0 = \cG$, $\cD_0 = \cD$, $A_0 = \{x \in X : \deg_{\cD_0}^-(x) \leq n\}$, and $A_{-1} = B_{-1} = \emptyset$. 
    We recursively define $B_m = \bigcup_{k \leq m} A_k$, $\G_{m+1} = \G \rest (X \setminus B_{m})$, $\cD_{m+1} = \cD \rest (X \setminus B_{m})$, and $A_{m + 1} = \{x \in X \setminus B_m : \deg_{\cD_{m + 1}}^-(x) \leq n \}$. We claim that the following conditions hold:
    \begin{enumerate}
        \item $A_m \cap A_{m'} = \emptyset$ for all $m' \neq m$,
        \item $\mu(X \setminus B_m) \leq \frac{n}{n+1}\mu(X \setminus B_{m - 1})$, and 
        \item $\deg_{\cD_m}^-(x) \leq n$ for all $x \in A_m$.
    \end{enumerate}
    Conditions (1) and (3) are immediate from the construction, so we only need to prove (2).

    \begin{claim}
    \label{claim: ineq2}
        $\mu(X \setminus B_{m + 1}) \leq \frac{n}{n+1}\mu(X \setminus B_m)$ for all $m\in \N$.
    \end{claim}

    \begin{proof}[Proof of Claim~\ref{claim: ineq2}.]
        If not, then there exists $m\in \N$ such that $\mu(X \setminus B_{m + 1}) > \frac{n}{n+1}\mu(X \setminus B_m)$. So
        \begin{align*}
            \deg_{\av}^-(\cD_{m + 1}) & = \frac{1}{\mu(X \setminus B_m)} \int_{X \setminus B_m} \deg_{\cD_{m + 1}}^-(x) \dmu(x) \\
            & \geq \frac{1}{\mu(X \setminus B_m)}\int_{X \setminus B_{m + 1}} \deg_{\cD_{m + 1}}^-(x)\dmu(x) \\
            & \geq \frac{1}{\mu(X \setminus B_m)}\mu(X \setminus B_{m + 1})(n + 1) \\
            & > \frac{1}{\mu(X \setminus B_m)}\Big(\mu(X \setminus B_m)\frac{n}{n + 1}\Big)(n + 1) \\
            & = n,
        \end{align*}
        contradicting that $\deg_{\av}^-(\cD_{m + 1}) \leq n$. This proves the claim.
    \end{proof}

    Finally, from (3) above, it follows that for all $x\in A_m$, $\deg_{\G_m}(x) = \deg_{\cD_m}^+(x) + \deg_{\cD_m}^-(x)\leq 2n$. Combining this with (1) and (2), we thus deduce from Lemma~\ref{lem: list coloring from r} that $\chi_{\mu}^{\ap}(\G) \leq 2n + 1$, as desired.
\end{proof}

\begin{remark}
    \label{rem: digraph vs functions}
    The statement of Theorem~\ref{thm:nfunctions} is in fact equivalent to the following statement: Let $(X, \mu)$ be a standard probability space, and let $\cD$ be a bounded-degree Borel directed graph on $X$ such that the underlying undirected graph $\G$ of $\cD$ is pmp. If $\deg_{\cD}^+(x) \leq n$ for each $x \in X$, then $\chi_{\mu}^{\ap}(\G) \leq 2n + 1$. This follows immediately from the observation that any such $\cD$ is generated by $n$ bounded-to-one Borel functions; see \cite[Section~5.3]{km2020}.
\end{remark}

\begin{remark}
\label{rem: paley}
    It is well-known that the bound in Theorem~\ref{thm:nfunctions} is best possible. For instance, consider the \emph{Paley tournament} $\G$ on $\{0, \dots, 6 \}$ generated by the functions 
    \begin{align*}
        f_1(x) & = x + 1 \mod{7}, \\
        f_2(x) & = x + 2 \mod{7}, \\
        f_3(x) & = x + 4 \mod{7}.
    \end{align*}
    Then it is easy to check that $\chi(\G) = 2 \cdot 3 + 1 = 7$.
\end{remark}

\section{Lower bounds on the chromatic number}\label{sec:lowerbounds}

In this section, we give several new spectral lower bounds on the approximate measurable chromatic number of a bounded-degree Borel pmp graph. These bounds are connected to the following well-known classical theorem of Hoffman:

\begin{thm}
    \cite{hoffman1970} Let $\G$ be a finite graph. Then
    $$\chi(\G) \geq \ceil{1 - \frac{M(T_{\G})}{m(T_{\G})}}.$$
\end{thm}

Many of the proofs in this section closely resemble the corresponding classical proofs; see \cite[Chapter~19]{spielman2025}. However, especially the proof of the main result, Theorem~\ref{thm: hoffman bounded}, involves new ideas due to its approximate nature.

Recall that the \emph{$\mu$-independence number} $i_{\mu}(\G)$ of $\G$ is given by 
$$i_\mu(\G) = \sup\{\mu(A) : A \text{ is a $\mu$-measurable, $\G$-independent set} \}.$$ 
Since monochromatic sets in a coloring have to be independent, it is easy to see that, for any $\alpha \in (0, 1]$, if $i_{\mu}(\G) \leq \alpha$, then $\chi_{\mu}^{\ap}(\G) \geq \frac{1}{\alpha}$ (see also \cite{ck2013}).

We start with regular graphs, since the proofs in this case are considerably simpler. Our first main result is the following (compare with \cite[Proposition~4.16]{ck2013}):

\begin{thm}
\label{thm: hoffman regular}
    Let $\G$ be a $d$-regular Borel pmp graph on a standard probability space $(X, \mu)$. Then for any $\mu$-measurable, $\G$-independent set $S \sub X$, 
    $$\mu(S) \leq \frac{-m(T_\G)}{d - m(T_\G)} = 1 - \frac{d}{M(L_\G)}.$$
    In particular, $i_{\mu}(\G) \leq \frac{-m(T_\G)}{d - m(T_\G)}$.
\end{thm}

\begin{proof}
    We write $L=L_\G, T=T_\G$ for convenience. Let $S \sub X$ be $\mu$-measurable and $\G$-independent,
    and let $1_S$ be the characteristic function for $S$; note that $1_S$ is $\mu$-measurable. Set
    $$f = 1_S - \mu(S) \cdot 1.$$
    Notice that $L1 = 0$. Then since $L$ is self-adjoint,
    \begin{align*}
        \ang{Lf, f} & = \ang{L1_S, 1_S} \\
        & = \int_X \bigg(\deg(x) \cdot 1_S(x) - 1_S(x) \cdot \sum_{y \sim x} 1_S(y)\bigg) \cdot 1_S(x)\dmu(x) \\
        & = \int_S (\deg(x) - \deg_{\G \rest S}(x))\dmu(x) \\
        & = \int_S \deg(x)\dmu(x) \\
        & = d\mu(S)
    \end{align*}
    since $S$ is $\G$-independent. A simple calculation shows also that
    $$\ang{f, f} = \mu(S)(1 - \mu(S)).$$
    Hence,
    $$M(L) \geq \frac{\langle Lf,f\rangle}{\langle f,f\rangle} =  \frac{d\mu(S)}{\mu(S)(1 - \mu(S))} = \frac{d}{1 - \mu(S)}.$$
    We conclude that
    $\mu(S) \leq 1 - \frac{d}{M(L)} = \frac{M(L) - d}{M(L)} = \frac{-m(T)}{d - m(T)}.$
\end{proof}

\begin{cor}\label{cor:hoffman}
    Let $\G$ be a $d$-regular Borel pmp graph on a standard probability space $(X, \mu)$. Then
    $$\chi_{\mu}^{\ap}(\G) \geq \ceil{1 - \frac{d}{m(T)}}.$$
\end{cor}

Before showing that Corollary~\ref{cor:hoffman} in fact remains true for non-regular graphs (even though Theorem~\ref{thm: hoffman regular} in general does not), we note that for non-regular graphs, the exact same proof as for Theorem~\ref{thm: hoffman regular} gives the following slightly weaker statement: 

\begin{cor}
    Let $\G$ be a bounded-degree Borel pmp graph on a standard probability space $(X, \mu)$. Then for any $\mu$-measurable, $\G$-independent set $S \sub X$, 
    $$\mu(S) \leq 1 - \frac{(1/\mu(S))\int_S \deg(x)\dmu(x)}{M(L)}.$$
\end{cor}

In particular, if $\delta$ is the minimum degree of $\G$, then since $(1/\mu(S)) \int_S \deg(x)\dmu(x) \geq \delta$ for any set $S \sub X$, we deduce the following:

\begin{cor}
    Let $\G$ be a bounded-degree Borel pmp graph on a standard probability space $(X, \mu)$ with minimum degree $\delta$. Then
    $$
    i_\mu(\G) \leq 1 - \frac{\delta}{M(L)},
    $$
    and thus also
    $$
    \chi_{\mu}^{\ap}(\G) \geq \frac{M(L)}{M(L) - \delta}.
    $$
\end{cor}

We next turn to the main result of this section, Theorem~\ref{thm: hoffman bounded} below, which shows that Corollary~\ref{cor:hoffman} remains true in the general context of bounded-degree (as opposed to merely regular) graphs. We start with the following lemma:

\begin{lem}
\label{lem: block decomp ineq}
    Let $T$ be a bounded and self-adjoint operator on $L^2(X)$. Let $k \in \N$, and let $(A_i)_{1 \leq i \leq k}$ be a $\mu$-measurable partition of $X$. Then
    \begin{align*}
        (k - 1)m(T) + M(T) \leq \sum_{1 \leq i \leq k} M(T_{ii}),
    \end{align*}
    where $(T_{ij})_{i,j=1}^k$ is the block decomposition of $T$ corresponding to the partition $X = \sqcup_{i=1}^k A_i$.
\end{lem}

\begin{proof}
    The proof is by induction. If $k = 1$, then the result is clear. We separate the $k = 2$ case into a distinct claim.

    \begin{claim}
    \label{claim: k = 2}
        Suppose $X = A_1 \sqcup A_2$ is a $\mu$-measurable partition of $X$. Then $m(T) + M(T) \leq M(T_{11}) + M(T_{22})$.
    \end{claim}

    \begin{proof}[Proof of Claim~\ref{claim: k = 2}.]
        Let $\eps > 0$. Then there is a unit vector $f \in L^2(X)$ such that $M(T) - \eps < \ang{Tf, f}$. Since $L^2(X) = L^2(A_1) \oplus L^2(A_2)$, there are $f_1 \in L^2(A_1), f_2 \in L^2(A_2)$ such that $f = f_1 + f_2$.

        Suppose first that $N_1 = \norm{f_1}$ and $N_2 = \norm{f_2}$ are both nonzero. Set $g = \frac{N_2}{N_1}f_1 - \frac{N_1}{N_2}f_2$. Then
        \begin{align*}
            \ang{g, g} & = \bigang{\frac{N_2}{N_1}f_1 - \frac{N_1}{N_2}f_2, \frac{N_2}{N_1}f_1 - \frac{N_1}{N_2}f_2} \\
            & = \Big(\frac{N_2}{N_1}\Big)^2\ang{f_1, f_1} - \ang{f_1, f_2} - \ang{f_2, f_1} + \Big(\frac{N_1}{N_2}\Big)^2\ang{f_2, f_2} \\
            & = N_2^2 + N_1^2 = \norm{f}^2  = 1
        \end{align*}
        since $f$ is a unit vector. Therefore,
        \begin{align*}
            M(T) + m(T) & \leq (\ang{Tf, f} + \eps) + \ang{Tg, g} \\
            & = \ang{T_{11}f_1, f_1} + \ang{T_{12}f_2, f_1} + \ang{T_{21}f_1, f_2} + \ang{T_{22}f_2, f_2} +  \\
            & \hspace{1cm} \Big(\frac{N_2}{N_1}\Big)^2\ang{T_{11}f_1, f_1} - \ang{T_{12}f_2, f_1} - \ang{T_{21}f_1, f_2} + \Big(\frac{N_1}{N_2}\Big)^2\ang{T_{22}f_2, f_2} + \eps \\
            & = (1 + \Big(\frac{N_2}{N_1}\Big)^2)\ang{T_{11}f_1, f_1} + (1 + \Big(\frac{N_1}{N_2}\Big)^2)\ang{T_{22}f_2, f_2} + \eps \\
            & = \frac{\ang{T_{11}f_1, f_1}}{\ang{f_1, f_1}} + \frac{\ang{T_{22}f_2, f_2}}{\ang{f_2, f_2}} + \eps \\
            & \leq M(T_{11}) + M(T_{22}) + \eps.
        \end{align*}
        Since this inequality holds for all $\eps > 0$, we indeed have $M(T) + m(T) \leq M(T_{11}) + M(T_{22})$ in this first case.

        Next, assume that $\norm{f_2} = 0$. (If $\norm{f_1} = 0$ instead, then a similar argument applies.) Then $\ang{Tf, f} = \ang{Tf_1, f_1} = \ang{T_{11}f_1, f_1}$. So,
        $$M(T) - \eps \leq \ang{Tf, f} = \ang{T_{11}f_1, f_1} \leq M(T_{11}),$$
        since $\norm{f_1} = \norm{f} = 1$. Then since also
        $$m(T) \leq m(T_{22}) \leq M(T_{22})$$
        by Proposition~\ref{prop: block decomp T}, we again have
        $$M(T) + m(T) \leq M(T_{11}) + \eps + M(T_{22})$$
        for all $\eps > 0$. This concludes the proof of the claim.
    \end{proof}

    Now let $k \geq 2$, and assume the inequality in the statement of the lemma holds for $k - 1$. Let $T$ be as in the statement of the lemma. Let moreover $S$ be the operator appearing in the $(1, 1)$-entry in the block decomposition of $T$ over the measurable partition of $X$ into the two sets $A_1 \sqcup \cdots \sqcup A_{k - 1}$ and $A_k$, and note that $T_{kk}$ is the $(2, 2)$-entry in this decomposition. By Claim~\ref{claim: k = 2},
    $$M(T) + m(T) \leq M(S) + M(T_{kk}).$$
    By the inductive hypothesis,
    $$(k - 2)m(S) + M(S) \leq \sum_{1 \leq i \leq k - 1} M(T_{ii}).$$
    So,
    \begin{align*}
        M(T) + (k - 1)m(T) & \leq M(T) + m(T) + (k - 2)m(S) \\
        & \leq M(S) + M(T_{kk}) + (k - 2)m(S) \\
        & \leq \sum_{1 \leq i \leq k} M(T_{ii}),
    \end{align*}
    where the first inequality follows from Proposition~\ref{prop: block decomp T}. This finishes the proof of the Lemma.
\end{proof}

We now turn to the main theorem of this section (Theorem~\ref{thm:hoffman_intro} from the introduction):

\begin{thm}
\label{thm: hoffman bounded}
    Let $\G$ be a bounded-degree Borel pmp graph on a standard probability space $(X, \mu)$. Then $\chi_{\mu}^{\ap}(\G) \geq \ceil{1 - \frac{M(T_{\G})}{m(T_{\G})}}$.
\end{thm}

\begin{proof}
    Let $T = T_{\G}$, and let $k = \chi_{\mu}^{\ap}(\G)$. Then for any $\eps > 0$, there are $\mu$-measurable, pairwise disjoint, $\G$-independent sets $A^{\eps}_1, \dots, A^{\eps}_k \sub X$ such that $\mu(\bigcup_{1 \leq i \leq k} A^{\eps}_i) > 1 - \eps$. Let $B^{\eps} = X \setminus (\bigcup_{1 \leq i \leq k} A^{\eps}_i)$. Then $T$ admits a block decomposition $(T^{\eps}_{ij})_{1 \leq i, j \leq k + 1}$ corresponding to this partition; we also write $T^{\eps}_B$ for the $(k + 1, k + 1)$-entry, rather than $T^{\eps}_{k+1,k+1}$. Additionally, we consider the block decomposition of $T$ corresponding to the partition of $X$ into the two sets $A^{\eps}_1 \sqcup \cdots \sqcup A^{\eps}_k$ and $B^{\eps}$; we write $S^{\eps}$ for the $(1, 1)$-entry of this decomposition, $R^{\eps}_{12}$ for the $(1, 2)$-entry, and $R^{\eps}_{21}$ for the $(2, 1)$-entry; note that the $(2, 2)$-entry is $T^{\eps}_B$.

    For each $n \in \N$, let $\eps_n = \frac{1}{n}$. We claim that $\lim_{n \to \infty} M(S^{\eps_n}) = M(T)$. To see this, fix $\eps > 0$, and let
    $$\eps_0 = \frac{\eps}{3\norm{T} + 1}.$$
    Then there is a unit vector $f \in L^2(X)$ such that $\ang{Tf, f} > M(T) - \eps_0$. By Lemma~\ref{lem: dct}, there is $\delta > 0$ such that, for any measurable $B \sub X$, $\mu(B) < \delta$ implies $\int_B \abs{f(x)}^2\dmu(x) < \eps_0$. Let $N \in \N$ be such that $\frac{1}{N} < \delta$. Let $n \geq N$, and let $f_1 \in L^2(A^{\eps_n}_1 \sqcup \cdots \sqcup A^{\eps_n}_k)$, $f_2 \in L^2(B^{\eps_n})$ be such that $f = f_1 + f_2$. Then by combining this with Proposition~\ref{prop: block decomp T}(2), we get
    \begin{align*}
        \abs{\ang{Tf, f} - \ang{S^{\eps_n}f_1, f_1}} & \leq \abs{\ang{R^{\eps_n}_{12}f_2, f_1}} + \abs{\ang{R^{\eps_n}_{21}f_1, f_2}} + \abs{\ang{T^{\eps_n}_Bf_2, f_2}} \\
        & \leq \norm{R^{\eps_n}_{12}}\norm{f_2}\norm{f_1} + \norm{R^{\eps_n}_{21}}\norm{f_1}\norm{f_2} + \norm{T^{\eps_n}_B}\norm{f_2}^2 \\
        & \leq 3\norm{T}\norm{f_2} \leq 3\norm{T}\eps_0.
    \end{align*}
    So, $M(S^{\eps_n}) \geq \ang{Tf, f} - 3\norm{T}\eps_0 \geq M(T) - (3\norm{T} + 1)\eps_0 = M(T) - \eps$. Since also $M(S^{\eps_n}) \leq M(T)$ by Proposition~\ref{prop: block decomp T}(1), this shows that, indeed, $\lim_{n \to \infty} M(S^{\eps_n}) = M(T)$.

    Finally, we claim 
    $$(k - 1)m(T) + M(T) \leq 0.$$
    Fix $\eps > 0$, and let $\eps' > 0$ be such that $M(T) - M(S^{\eps'}) < \eps$. By Lemma~\ref{lem: block decomp ineq},
    $$(k - 1)m(S^{\eps'}) + M(S^{\eps'}) \leq \sum_{1 \leq i \leq k} M(T^{\eps'}_{ii}) = 0$$
    since $A_i$ is $\G$-independent for each $i \leq k$. So
    $$(k - 1)m(T) + M(T) \leq (k - 1)m(S^{\eps'}) + (M(S^{\eps'}) + \eps) = \eps.$$
    Therefore, using that $m(T)<0$ (see Lemma~\ref{lem:mnegative}),
    $$
    \chi_{\mu}^{\ap}(\G) = k \geq 1 - \frac{M(T)}{m(T)}.
    $$
    This finishes the proof of the theorem.
\end{proof}

\section{Matchings}\label{sec:matchings}
In this section, we prove that any bounded-degree Borel pmp graph whose spectral values meet a certain condition satisfies a property that we call the \emph{strict measurable Tutte condition}. Our main result, Theorem~\ref{thm: brouwer haemers}, is inspired by the following classical theorem of Brouwer and Haemers:

\begin{thm}
    \cite[Theorem~2.3]{BH05} Let $\G$ be a finite graph, and let $0 = \lambda_1 \leq \lambda_2 \leq \cdots \leq \lambda_n$ be the eigenvalues of the Laplacian $L_{\G}$. If $n$ is even and $2\lambda_2 \geq \lambda_n$, then $\G$ admits a perfect matching.
\end{thm}

For the remainder of this section, we fix an ergodic $d$-regular Borel pmp graph $\G$ with spectral gap. We denote by $L=L_{\G}$ the Laplacian operator, and define
$$
m_L = \inf_{0\neq f\in L^2_0(X)} \frac{\langle Lf,f\rangle}{\langle f,f\rangle} \quad \text{and}\quad M_L = \sup_{0\neq f\in L^2_0(X)} \frac{\langle Lf,f\rangle}{\langle f,f\rangle}.
$$
From the basic results in Sections~\ref{subsec:operators} and \ref{subsec: spectral gap}, we note that $0<m_L\leq M_L = M(L)\leq 2d$, $\sigma(L)\sub \{0\}\cup [m_L,M_L]$, and the eigenspace of $0$ is equal to $\C1$. Note that $m_L>0$ is equivalent to $\G$ having spectral gap; in particular, the condition $2m_L\geq M_L$, appearing in Theorem~\ref{thm: brouwer haemers} below, automatically implies spectral gap.

We start with proving the following spectral inequality for subsets of $X$, which is the direct generalization to Borel pmp graphs of \cite[Lemma~6.1]{Hae95}:

\begin{prop}\label{prop:XYZ}
   
    Let $Y$ and $Z$ be disjoint subsets of $X$ of positive measure such that there are no edges between $Y$ and $Z$. Then
    $$
    \frac{\mu(Y)\mu(Z)}{(1-\mu(Y))(1-\mu(Z))} \leq \left(\frac{M_L - m_L}{M_L + m_L}\right)^2.
    $$
\end{prop}
\begin{proof}
We write $M = M_L$ and $m = m_L$ for convenience, and we define $a = \frac{M+m}{2}$. Consider the operator
$$
A = \begin{pmatrix}
    0 & L - a I\\
    L - a I & 0
\end{pmatrix}
$$
on $L^2(X)\oplus L^2(X)$. From Lemma~\ref{lem:antidiagspectrum}, we get
$$
\sigma(A) = \sigma(L - a I) \cup - \sigma(L - a I) \sub \{-a\} \cup \left[-\frac{M-m}{2}, \frac{M-m}{2}\right] \cup \{a\}.
$$
Similar to Proposition~\ref{prop:isolated} and Corollary~\ref{cor:ergodicspectralgapspectrum}, we see that the eigenspaces of $-a$ and $a$ equal $\C (1\oplus 1)$ and $\C (-1\oplus 1)$, respectively, and
$$
\sigma\left(A|_{L^2(X)\ominus (\C (1\oplus 1)\oplus \C (-1\oplus 1))}\right) \sub \left[-\frac{M-m}{2}, \frac{M-m}{2}\right].
$$
Consider the decomposition 
\begin{equation}\label{eq:XdecompYZ}
    L^2(X) \oplus L^2(X) = L^2(Z) \oplus L^2(X\setminus Z) \oplus L^2(X\setminus Y) \oplus L^2(Y).
\end{equation}
Since there are no edges between $Y$ and $Z$, the corresponding block decomposition of $A$ has the form
$$
A = \begin{pmatrix}
    0 & 0 & \ast & 0 \\
    0 & 0 & \ast & \ast \\
    \ast & \ast & 0 & 0 \\
    0 & \ast & 0 & 0 
\end{pmatrix}.
$$
For notational convenience, we also write $X_1 = Z, X_2 = X\setminus Z, X_3 = X\setminus Y, X_4 = Y$, so that the decomposition~\eqref{eq:XdecompYZ} becomes $L^2(X) \oplus L^2(X) = \oplus_{i=1}^4 L^2(X_i)$. Additionally, we write $v_i = 1_{X_i}$, i.e., the vector that is the constant function $1$ on $L^2(X_i)$ and $0$ elsewhere. Using this notation, we then construct the $4\times 4$ matrix $B$ with entries
$$
b_{ij} = \frac{1}{\mu(X_i)} \langle v_i, Av_j\rangle.
$$
We compute
\begin{align*}
    b_{2,4} &= \frac{1}{\mu(X_2)} \langle v_2, A v_4\rangle\\
    &= \frac{1}{\mu(X\setminus Z)} \int_X 1_{X\setminus Z}(x) (\deg(x) 1_Y(x) - \sum_{y\sim x} 1_Y(y) - a 1_Y(x)) \dmu(x)\\
    &= \frac{1}{1-\mu(Z)} \left(\int_Y \deg(x) \dmu(x) - \int_{X\setminus Z} \sum_{y\sim x} 1_Y(y) \dmu(x) - a \mu(Y)\right) \\
    &= -a \frac{\mu(Y)}{1-\mu(Z)}.
\end{align*}
In the last line, we use the fact that there are no edges between $Y$ and $Z$, together with the fact that $\int_Y \deg(x) \dmu(x) = \int_X \sum_{y\sim x} 1_Y(y) \dmu(x)$; this follows from line~\eqref{ineq: n(a)}, with $A$ replaced by $Y$. Since $b_{21} = b_{22} = 0$ and the second row of $B$ has to sum to
$$
\frac{1}{1-\mu(Z)} \langle 1_{X\setminus Z}, (L-aI) 1\rangle = -a,
$$
we immediately get that $b_{23} = -a + a \frac{\mu(Y)}{1-\mu(Z)}$. Using similar computations, it is then easy to check that $B$ is given by
$$
B = \begin{pmatrix}
    0 & 0 & -a & 0 \\
    0 & 0 & -a + a \frac{\mu(Y)}{(1-\mu(Z))} & -a \frac{\mu(Y)}{(1-\mu(Z))} \\
    - a \frac{\mu(Z)}{(1-\mu(Y))} & -a + a \frac{\mu(Z)}{(1-\mu(Y))} & 0 & 0 \\
    0 & -a & 0 & 0 
\end{pmatrix}.
$$
A linear algebra exercise then tells us that the eigenvalues of $B$ are given by $-a < - \lambda < \lambda < a$ for some $\lambda > 0$ (namely $\lambda = \frac{a\sqrt{\mu(Y)\mu(Z)}}{\sqrt{(1-\mu(Y))(1-\mu(Z))}}$). 

\begin{claim} 
\label{claim: M-m}
$\lambda \leq \frac{M-m}{2}$.
\end{claim}

\begin{proof}[Proof of Claim~\ref{claim: M-m}] Denote by 
$$
K = \begin{pmatrix}
    \sqrt{\mu(X_1)} &0&0&0\\
    0& \sqrt{\mu(X_2)} &0&0\\
    0&0& \sqrt{\mu(X_3)} &0\\
    0&0&0& \sqrt{\mu(X_4)}
\end{pmatrix}
$$
and let $\tilde B = KBK^{-1}$. Then it is clear that $\tilde B$ and $B$ have the same eigenvalues. Moreover, denoting by
$$
S: \C^4 \to \oplus_{i=1}^4 L^2(X_i): (t_1, t_2, t_3, t_4) \mapsto \oplus_{i=1}^4 \frac{1}{\sqrt{\mu(X_i)}} t_i 1_{X_i},
$$
it is easy to check that the adjoint of $S$ is given by 
$$
S^*(f) = \left(\frac{1}{\sqrt{\mu(X_1)}} \int_{X_1} f(x)\dmu(x), \ldots, \frac{1}{\sqrt{\mu(X_4)}} \int_{X_4} f(x)\dmu(x)\right),
$$
that $S^*S=1$, i.e., $S$ is an isometry, and that $\tilde B = S^*AS$. Since one can easily check that the conditions from Lemma~\ref{lem:interlacing} are satisfied, we indeed conclude that $\lambda\leq \frac{M-m}{2}$.
\end{proof}

Finally, from the claim, it follows immediately that
$$
a^2 \frac{\mu(Y)\mu(Z)}{(1-\mu(Y))(1-\mu(Z))} = \frac{\det(B)}{a^2} = \lambda^2 \leq \left(\frac{M-m}{2}\right)^2.
$$
Since $a = \frac{M+m}{2}$, dividing both sides by $a^2$ gives the desired result.
\end{proof}

\subsection{Tutte's condition} We now outline a condition in the pmp setting that is inspired by the classical Tutte's condition for the existence of a perfect matching. 

Let $A \sub X$ be a measurable subset, and let
$$
\cO_A = \{x\in X\setminus A : x \text{ is contained in a finite odd component of $\cG \rest (X\setminus A)$}\}.
$$
Note that
$$
\cO_A = \bigcup_{k\geq 0} \cO_A^{(2k+1)},
$$
where $\cO_A^{(2k+1)}$ denotes the points in $X\setminus A$ contained in odd components of size $2k+1$. We note here that $\cO_A^{(1)}$ is $\G$-independent by construction.

Define a measure $\nu_{A,k}$ on $\cO_A^{(2k+1)}$ by $\nu_{A,k} = \frac{1}{2k+1}\mu$, and let
$$
\nu_A = \oplus \nu_{A,k},
$$
i.e., for $Y\sub \cO_A$, 
$$
\nu_A(Y) = \sum_{k\geq 0} \nu_{A,k}(Y\cap \cO_A^{(2k+1)}) = \sum_{k\geq 0} \frac{1}{2k+1} \mu(Y\cap \cO_A^{(2k+1)}).
$$
Intuitively, $\nu_A$ ``counts'' the number of finite odd components that arise after $A$ is removed from the vertex set. We now state our suggestion for a measurable (expanding) version of Tutte's condition.

\begin{defn}
\label{defn: tutte}
    We say $\G$ satisfies the \textit{strict measurable Tutte condition} if there exists $c<1$ such that for every measurable subset $A\sub X$, the set 
    $$
    \cO_A \coloneqq \{x\in X\setminus A : x \text{ is contained in a finite odd component of $\cG \rest (X\setminus A)$}\}
    $$
    satisfies $\nu_A(\cO_A) \leq c\mu(A)$.
\end{defn} 

We point out that, for finite graphs, Tutte's condition is the same as the above with $c=1$.

In Theorem~\ref{thm: brouwer haemers} below, we show that the spectral condition analogous to the one in the classical Brouwer-Haemers theorem implies the strict measurable Tutte condition. Before we embark on its proof, we first prove two useful lemmas.

\begin{lem}\label{lem:OAmu}
    Assume $\mu$ is nonatomic. Then for any measurable $A\sub X$ and $0\leq t\leq \mu(\cO_A)$, there exists a subset $B\sub \cO_A$ with measure $\mu(B) = t$ that is a union of components of $\G \rest \cO_A$.
\end{lem}
\begin{proof}
Consider the Borel graph $\cH = \cG \rest \cO_A$. Since every component of $\cH$ is finite, we deduce that $E_{\cH}$ is a finite Borel equivalence relation; so in particular, there exists a Borel transversal $\cT\sub \cO_A$ meeting every component of $\cH$ exactly once.

If $t=\mu(\cO_A)$, we can obviously take $B=\cO_A$ to get the result, so assume $t < \mu(\cO_A)$. Since $\cO_A$ is equal to the disjoint union $\cO_A = \bigcup_{k\geq 0} \cO_A^{(2k+1)}$, there exists a unique $K\in \N$ such that
$$
\mu\big(\cup_{0\leq k\leq K-1} \cO_A^{(2k+1)}\big) \leq t < \mu\big(\cup_{0\leq k\leq K} \cO_A^{(2k+1)}\big).
$$
Since $\cT\cap \cO_A^{(2K+1)}$ is Borel and $\mu$ is nonatomic, we can find a Borel subset $\cT_0\sub \cT \cap \cO_A^{(2K + 1)}$ of $\mu$-measure 
$$
\frac{1}{2K+1} \big(t - \mu\big(\cup_{0\leq k\leq K-1} \cO_A^{(2k+1)}\big)\big).
$$
Denoting by $T_0$ the $E_{\cH}$-saturation of $\cT_0$, we find that $B = \cup_{0\leq k\leq K-1} \cO_A^{(2k+1)} \cup T_0$ satisfies the requirements, finishing the proof.
\end{proof}

\begin{lem}\label{lem:2mMindependence}
    Assume $2m_L \geq M_L$ and $\mu$ is nonatomic. Then the measurable independence number of $\G$ satisfies $i_\mu(\G)<\frac{1}{2}$. 
\end{lem}
\begin{proof}
We will show that under the assumptions of the lemma, we necessarily have $m(T_{\G}) > -d$. Once we know this, the conclusion follows immediately from Theorem~\ref{thm: hoffman regular}. 

Thus, we assume for contradiction that $m(T_{\G}) = -d$. It follows that $M_L = M(L_{\G}) = d - m(T_{\G}) = 2d$. Since $2m_L \geq M_L$, we conclude that $m_L \geq d$, and in particular, $\sigma(T_{\G}) \sub [-d,0]\cup \{d\}$. From Corollary~\ref{cor:bipartite}, we then deduce that $\{-d,d\}\sub \sigma(T_{\G}) \sub \{-d, 0, d\}$. In particular, Proposition~\ref{prop:isolated} implies that both $d$ and $-d$ are eigenvalues. Corollary~\ref{cor:ergodicspectralgapspectrum} moreover tells us that the eigenspace of $d$ is equal to $\C1$. Let $f$ be an eigenfunction of $-d$, and assume without loss of generality that $f$ is real-valued. Then the same argument as in the proof of Claim~\ref{claim: constant} implies that the absolute value of $f$ is constant; furthermore, as in the proof of Theorem~\ref{thm:-dspecimpliesbipartite}, if we let $A = \{x\in X : f(x) < 0\}$, then $A$ and $B = X\setminus A$ yield a measurable bipartition of $X$. 

We claim that any other eigenfunction $g$ of $-d$ is a multiple of $f$, i.e., the eigenspace of $-d$ is $1$-dimensional. Indeed, assume not, and let $A'$ and $B'$ be the parts of the measurable bipartition associated to $g$. By assumption, both $A_0 = A'\cap A$ and $B_0 = A'\cap B$ have positive measure, as do $A_1 = B'\cap A = A\setminus A_0$ and $B_1 = B'\cap B = B\setminus B_0$. However, by construction, there are no edges between $A_0$ and $B_0$ (as they are subsets of $A'$), and similarly there are no edges between $A_0$ and $A_1$, between $B_0$ and $B_1$, and between $A_1$ and $B_1$. In particular, $A_0\cup B_1$ and $A_1\cup B_0$ have no edges between them, and thus are both invariant sets of positive measure, contradicting ergodicity.

We conclude that $L^2(X) = \C1 \oplus \C f \oplus H$, where $f$ is the eigenfunction of $-d$ yielding the bipartition $X=A\cup B$, and $T_{\G} \rest H = 0$. Taking a subset $B_0\sub B$ with $\mu(B_0) = \alpha\mu(B)$, for some irrational $0 < \alpha < 1$, we can now consider the function 
$$
g = \mu(B)1_{B_0} - \mu(B_0)1_B\in H.
$$
Then for each $x\in A$, we have 
$$
(T_{\G}g)(x) = \sum_{y\sim x} g(y) = \sum_{y\sim x} \big(\mu(B)1_{B_0}(y) - \mu(B_0)1_B(y)\big) = \sum_{y\sim x} \mu(B)1_{B_0}(y) - d\mu(B_0),
$$
since all neighbors of $x\in A$ lie in $B$ and $\deg(x)=d$. As $T_{\G}\rest H = 0$, we conclude that for every $x\in A$,
$$
\sum_{y\sim x} \mu(B)1_{B_0}(y) = d\mu(B_0),
$$
and hence
$$
\sum_{y\sim x} 1_{B_0}(y) = d\alpha.
$$
Since $\alpha$ is irrational, this yields the desired contradiction.
\end{proof}

We are now ready for the main theorem of this section (Theorem~\ref{thm:matching_intro} from the introduction).

\begin{thm}
\label{thm: brouwer haemers}
Let $\G$ be an ergodic $d$-regular Borel pmp graph on a standard probability space $(X, \mu)$. If $2m_L \geq M_L$, then the strict measurable Tutte condition holds for $\G$.
\end{thm}
\begin{proof}
We again write $M = M_L$ and $m = m_L$. If $\mu$ is atomic, then ergodicity and the fact that $\G$ is pmp imply that we can reduce to the case of a finite connected graph, and thus \cite[Theorem~2.3]{BH05} applies. Hence, we can assume that $\mu$ is nonatomic.

Assume that $2m\geq M$ and that the strict measurable Tutte condition fails, i.e., for all $c<1$, there exists a measurable $A\sub X$ such that $\nu_A(\cO_A) > c\mu(A)$. We start with the following claim:

\begin{claim}
\label{claim: c}
There exists some $\frac{1}{2} < c < 1$ that admits an associated subset $A\sub X$ with $\nu_A(\cO_A) > c\mu(A)$ satisfying $\mu(A)<\frac{1}{2}$.
\end{claim}

\begin{proof}[Proof of Claim~\ref{claim: c}] Assume for contradiction that for every $c>\frac{1}{2}$, the associated set $A_c\sub X$ with $\nu_{A_c}(\cO_{A_c}) > c\mu(A_c)$ satisfies $\mu(A_c)\geq \frac{1}{2}$. Write $a_c = \mu(\cO_{A_c}^{(1)}) = \nu_{A_c}(\cO_{A_c}^{(1)})$ and $b_c = \nu_{A_c}(\cO_{A_c}\setminus \cO_{A_c}^{(1)}) \leq \frac{1}{3} \mu(\cO_{A_c}\setminus \cO_{A_c}^{(1)})$. The assumptions imply that
\begin{align*}
    \begin{cases}
        &a_c + b_c = \nu_{A_c}(\cO_{A_c}) > c\mu(A_c) \geq \frac{c}{2},\\
        &a_c + 3b_c \leq \mu(\cO_{A_c}) \leq 1 - \mu(A_c) \leq \frac{1}{2},
    \end{cases}
\end{align*}
An easy computation then yields 
$$
a_c > \frac{3c}{4} - \frac{1}{4}.
$$
Since $\cO_{A_c}^{(1)}$ is a $\G$-independent set, we thus have $i_\mu(\G) > \frac{3c}{4} - \frac{1}{4}$. Therefore,
$$
i_\mu(\G) \geq \sup_{\frac{1}{2} < c < 1} \Big(\frac{3c}{4} - \frac{1}{4}\Big) = \frac{1}{2}.
$$
Since $\mu$ is nonatomic and $2m\geq M$, Lemma~\ref{lem:2mMindependence} yields the desired contradiction.
\end{proof}

Using the claim, we now fix $\frac{1}{2} < c < 1$ and a corresponding subset $A$ satisfying $\nu_A(\cO_A) > c\mu(A)$ and $\mu(A) < \frac{1}{2}$. Write $s = \mu(A)$.

\textbf{Case 1.} $s \geq \frac{1}{1+2c}$. Then in particular $cs \geq \frac{1}{2}(1-s)$ and thus $\mu(\cO_A)\geq \nu_A(\cO_A) > c s \geq \frac{1}{2}(1-s)$. In particular, we see that
$$
\mu(X\setminus (A\cup \cO_A)) = \mu(X\setminus A) - \mu(\cO_A) <  \frac{1}{2}(1-s).
$$
Using Lemma~\ref{lem:OAmu}, we can thus define
$$
Y = [X\setminus (A\cup \cO_A)] \cup Y_0,
$$
where $Y_0$ is a union of components of $\cO_A$ and $\mu(Y_0) = \frac{1}{2}(1 - s) - \mu(X \setminus (A \cup \cO_A))$, so that $\mu(Y) = \frac{1}{2}(1-s)$. Let $Z=\cO_A\setminus Y_0$. Then $A\cup Y\cup Z = X$, $\mu(Z) = \mu(Y) = \frac{1}{2}(1-s)$, and by construction, there are no edges between $Y$ and $Z$. Proposition~\ref{prop:XYZ} thus implies that 
$$
\left(\frac{M-m}{M+m}\right)^2 \geq \frac{\mu(Y)\mu(Z)}{(1 - \mu(Y))(1 - \mu(Z))} = \frac{\mu(Y)\mu(Z)}{s + \mu(Y)\mu(Z)} = \frac{(1-s)^2}{(1+s)^2}.
$$
Since $s < \frac{1}{2}$, we get
$$
\frac{M-m}{M+m} \geq \frac{1-s}{1+s} > \frac{\frac{1}{2}}{\frac{3}{2}} = \frac{1}{3}.
$$
Hence $2m<M$, contradicting the assumption from the theorem.

\textbf{Case 2.} $s < \frac{1}{1+2c}$.  Since in this case $cs < \frac{1}{2}(1-s)$ and $\mu(\cO_A)\geq \nu_A(\cO_A) > c s$, by using Lemma~\ref{lem:OAmu} to add components of $\cO_A$ to $X\setminus (A\cup \cO_A)$ if necessary, we can find subsets $Y,Z\sub X$ without edges between them such that $X = A\cup Y\cup Z$ and $\min\{\mu(Y),\mu(Z)\} > c s$. In particular, $\mu(Y)\mu(Z) > cs(1-s-cs)$. Using that the function $x\mapsto \frac{x}{s+x}$ is strictly increasing, Proposition~\ref{prop:XYZ} thus implies
$$
\left(\frac{M-m}{M+m}\right)^2 \geq \frac{\mu(Y)\mu(Z)}{s + \mu(Y)\mu(Z)} > \frac{cs(1-s-cs)}{s+cs(1-s-cs)} = \frac{c}{1+c}\left(1 - \frac{s}{1-cs}\right).
$$
Since the function $x\mapsto \frac{x}{1-cx}$ is increasing as well, and $s < \frac{1}{1+2c}$, we deduce that
$$
\left(\frac{M-m}{M+m}\right)^2 \geq \frac{c}{1+c}\left(1- \frac{\frac{1}{1+2c}}{1-c\cdot \frac{1}{1+2c}}\right) = \left(\frac{c}{1+c}\right)^2.
$$
We conclude that
$$
\frac{M-m}{M+m} \geq \frac{c}{1+c} > \frac{1}{3},
$$
where we use that $c > \frac{1}{2}$. Therefore, also in this second case, we get $2m<M$, contradicting the original assumption. This finishes the proof.
\end{proof}

Finally, similarly to the classical fact that Tutte's condition implies Hall's condition for bipartite graphs, we observe next that for bipartite (i.e., odd-cycle-free) Borel pmp graphs, the strict measurable Tutte condition implies strict expansion for independent sets. Here, using the terminology of \cite{km2020}, we say that a Borel graph $\G$ on a standard probability space $(X, \mu)$ is \emph{strictly expanding for independent sets} if there exists $c > 1$ such that, for each measurable $\G$-independent set $A \sub X$, $\mu(\cN(A)) \geq c\mu(A)$.

\begin{lem}\label{lem:matchingbipartite}
    Let $\G$ be an ergodic $d$-regular Borel pmp graph on a standard probability space $(X, \mu)$ and assume $\G$ is odd-cycle-free. Then the strict measurable Tutte condition implies strict expansion for independent sets.
\end{lem}

\begin{proof}
    We prove the contrapositive. Suppose $\G$ is not strictly expanding for independent sets, and let $c<1$.

    By assumption, there is a measurable $\G$-independent set $A \sub X$ such that $\mu(\cN(A)) < (1/c)\mu(A)$. Then since $A$ is $\G$-independent, each element of $A$ is an isolated vertex in $\G \rest (X \setminus \cN(A))$. In particular, there is a subset of $X \setminus \cN(A)$ of measure at least $\mu(A)$ that is a disjoint union of $\G \rest (X \setminus \cN(A))$-components of cardinality $1$. Therefore,
    $$\nu_{\cN(A)}(\cO_{\cN(A)}) \geq \mu(A) > c\mu(\cN(A)).$$

    Since $c<1$ was arbitrary, this shows that the strict measurable Tutte condition is not satisfied. This finishes the proof.
\end{proof}

In particular, by combining Theorem~\ref{thm: brouwer haemers}, Lemma~\ref{lem:matchingbipartite}, and \cite[Remark~2.6]{ln2011}, we deduce that the assumption $2m_L \geq M_L$ implies the existence of a measurable perfect matching for \emph{bipartite} ergodic regular Borel pmp graphs.\footnote{We note that the definition of strict expansion used in \cite{ln2011} differs from the definition we use here, but for the proof of \cite[Remark~2.6]{ln2011}, the assumption of strict expansion for independent sets is sufficient.} However, we note that by combining \cite[Proposition~11.1]{HLS14} with \cite[Remark~2.6]{ln2011}, one can show immediately that any such graph that has spectral gap already admits a measurable perfect matching.

\begin{remark}
    If one could prove that the strict measurable Tutte condition implies the existence of a measurable perfect matching, the conclusion of Theorem~\ref{thm: brouwer haemers} would read, ``\dots then $\G$ admits a measurable perfect matching'', which would be an interesting new result for general pmp graphs. However, while we are not aware of any counterexamples (in the case of strict expansion; otherwise, see \cite{kun2024}), we also do not know of any positive results proving measurable versions of Tutte's theorem. This is in contrast with the Baire-measurable setting (see \cite{kl2023}).
\end{remark}

\section{Examples from local-global convergence}\label{sec:examples}

In general, it is difficult to compute the spectrum of the adjacency operator of a given bounded-degree Borel pmp graph $\G$. However, when $\G$ arises as a suitable limit of finite graphs, it often becomes easier to compute (properties of) the spectrum of $\G$ from those of the approximating finite graphs. A well-studied notion of convergence for which this is possible is \textit{local-global convergence}. This notion was introduced by Bollob\'{a}s and Riordan \cite{BR11} under the name partition metric, while the identification of a bounded-degree Borel pmp graph as limit object was carried out later by Hatami, Lov\'{a}sz, and Szegedy \cite{HLS14}, who also obtained most currently known results. More precisely, the limit object of a sequence of finite graphs $\cG_n$ with uniformly bounded degree that converges in the local-global sense is a weak equivalence class of bounded-degree Borel pmp graphs (see \cite{kechris2010, CGdlS21}). While the limit is thus not determined up to isomorphism, all limits of the sequence $(\G_n)_{n \in \N}$ share various features, and in particular their adjacency operators have the same spectrum.

We may identify the space $G_d(X,\mu)$ of Borel pmp graphs of degree bounded by $d$ on a standard probability space $(X,\mu)$ with a subspace of $A(\bF_{2d-1},X,\mu)$, the standard Borel space of all Borel measure-preserving actions of the free group of order $2d-1$ on $(X,\mu)$ (see for instance \cite{kechris2010}). Indeed, every Borel pmp graph of with maximum degree at most $d$ admits a Borel $(2d-1)$-edge coloring (see \cite[Proposition~4.6]{kst1999}), and is thus generated by an action of $\bF_{2d-1}$. We can now consider the topology on $A(\bF_{2d-1},X,\mu)$ induced by local-global convergence. One can show that this topology is induced by a pseudo-metric whose zero classes are weak equivalence classes, and that the metric quotient is compact \cite{AE11}. Since having maximum degree at most $d$ is preserved under local-global convergence, $G_d(X,\mu)\sub A(\bF_{2d-1},X,\mu)$ is closed, and hence its metric quotient is also compact.

Next, we give the formal definition of local-global convergence. For simplicity, we state it for sequences of finite graphs; for more details, we refer to \cite{HLS14}.
In the following, we assume that all our graphs have maximum degree at most $d$, where $d$ is a positive integer that we fix for the remainder of this section. We denote by $\cR\G_{r,k}$ the set of all (isomorphism classes of) rooted graphs of radius at most $r$ (and degree at most $d$), equipped with a vertex $k$-labeling. We write elements of $\cR\G_{r, k}$ in the form $(H, o, c)$, to represent the isomorphism class of the graph $H$ with root $o$ and $k$-labeling $c$. Also, we denote by $M(\cR\G_{r,k})$ the set of Borel probability measures on $\cR\G_{r,k}$, and we will usually equip it with the total variation distance given by $d_{\text{var}}(\mu,\nu) = \sup_{A\sub \cR\G_{r,k}} \abs{\mu(A) - \nu(A)}$. The next two definitions follow \cite{HLS14}.

\begin{defn}[Local statistics]
Let $G$ be a finite graph and $r$, $k$ be non-negative integers. For any vertex $k$-labeling $c:V(G)\to\{0,\dots,k-1\}$, we define the \textit{local statistics of $c$}, $P_{G,r}[c]$, to be the probability distribution on $\cR\G_{r,k}$ obtained by taking the $r$-neighborhood of a randomly selected vertex of $G$ equipped with the restriction of $c$. We denote the collection of local statistics by
\[
Q_{G,r,k} = \{P_{G,r}[c] : c \text{ is a vertex $k$-labeling}\} \sub M(\cR\G_{r,k}).
\]
When $\G$ is a bounded-degree Borel pmp graph on a standard probability space $(X,\mu)$ and $c:V(\G)\to \{0, \ldots, k-1\}$ is a measurable vertex $k$-labeling, we similarly define $P_{\G,r}[c]$ and $Q_{\cG,r,k}$, now choosing a vertex at random according to $\mu$.
\end{defn}

\begin{defn}[Local-global convergence]
We say a sequence of finite graphs $(G_n)_{n \in \N}$ {\em locally-globally converges} to some (not necessarily finite) Borel graph $\G$ on $(X,\mu)$ if for all non-negative integers $r$ and $k$, $Q_{G_n,r,k}$ converges to $Q_{\cG,r,k}$ in the Hausdorff metric as $n\to \infty$.
\end{defn}

Given a locally-globally convergent sequence of (finite) graphs, one may expect the spectrum of a limit graph to be exactly the set of limit points of the sequence of spectra, but care must be taken to ensure that a positive proportion of the mass of the corresponding eigenvectors does not concentrate on any null set in the limit. The following example, first observed by Riley Thornton, demonstrates this.

\begin{exmp}\label{eg:locglob}
    Let $G$ and $H$ be any two finite graphs such that there exists $\lambda\in\sigma(H)$, $\lambda\not\in \sigma(G)$, each equipped with the normalized counting measure. Let $G_n=G\sqcup H$ equipped with the measure $\mu_n$ which assigns $\frac{1}{n}$ of the mass uniformly to the vertices of $H$ and the remaining $1-\frac{1}{n}$ of the mass uniformly to the vertices of $G$. Then $(G_n,\mu_n)$ converges locally-globally to $G$, yet $\lambda$ is the in spectrum of each $G_n$ but not $G$.
\end{exmp}

To address this issue, we assume uniform integrability for the approximating eigenfunctions.

\begin{defn}
    Let $(X_n,\mu_n)$ be a sequence of (potentially finite) standard probability spaces. A sequence of functions $f_n: X_n\to \C$ is called \textit{uniformly integrable} if it satisfies
    $$
    \forall \eps > 0 \;\exists \delta > 0 \;\forall n\geq 1\;\forall A\sub X_n: \mu_n(A) < \delta \Rightarrow \norm{f_n \upharpoonright A}_2^2 < \eps.
    $$
    If $\cG$ is a bounded-degree Borel pmp graph on $(X,\mu)$, we say $\lambda\in \sigma(T_\G)$ admits \textit{uniformly integrable eigenfunctions} if we can find a uniformly integrable sequence of approximate unit eigenfunctions for $\lambda$ (where $X_n = X$ for each $n$). We say $\cG$ admits uniformly integrable eigenfunctions if every $\lambda\in \sigma(T_\G)$ does.
\end{defn}

We remark that, by Lemma~\ref{lem: dct}, any finite sequence of square-integrable functions is uniformly integrable.
Furthermore, it is well-known that uniform integrability of $(f_n)_n$ as defined above is equivalent to the fact that for every $\eps>0$, there exists $M>0$ such that 
$$
\sup_{n\in\N}\, \int_{\{x:\abs{f_n(x)} > M\}} \abs{f_n(x)}^2\dmu(x) < \eps.
$$
We note that this notion is usually called uniform square integrability, but since we will only be working with this one, we drop the word ``square'' for convenience.

In the following, we always work with probability measures. In particular, when writing $\ell^2(X)$ for some finite set $X$, we implicitly assume that we are using the normalized counting measure. The following result is the main theorem of this section, which in many cases characterizes the spectrum of a local-global limit. Note that while we state it for limit graphs on nonatomic standard probability spaces this is not much of a restriction, as the atomic (finite) parts are usually very easy to deal with separately.

\begin{thm}
\label{thm: locglob spectrum}
    Let $(G_m)_{m \in \N}$ be a sequence of finite graphs with uniformly bounded degree that locally-globally converges to a bounded-degree Borel pmp graph $\G$ on a nonatomic standard probability space $(X,\mu)$. Consider the following statement:  
    \begin{equation}
    \label{eq:uie}\tag{$\ast$}
    \parbox{\dimexpr\linewidth-4em}{
    \strut
    There exist a sequence $(m_n)$ and uniformly integrable unit vectors $f_{m_n}\in \ell^2(V(G_{m_n}))$  satisfying
    $\lim_{n\to\infty}\norm{T_{G_{m_n}}(f_{m_n})-\lambda f_{m_n}}_2=0$.
    \strut
    }
    \end{equation}
    Then the following hold:
    \begin{enumerate}
        \item If \eqref{eq:uie} holds for $\lambda\in\C$, then $\lambda\in \sigma(T_\G)$.
        \item If $\lambda\in\sigma(T_\G)$ admits uniformly integrable eigenfunctions, then \eqref{eq:uie} holds for $\lambda$.
    \end{enumerate}
    In particular, if $\G$ admits uniformly integrable eigenfunctions, then $\lambda\in \sigma(T_\G)$ if and only if \eqref{eq:uie} holds for $\lambda$.
\end{thm}
\begin{proof}
We start with (2). Assume $\lambda\in\sigma(T_\G)$ admits uniformly integrable eigenfunctions, i.e., there exists a uniformly integrable sequence of unit vectors $\xi_n\in L^2(X)$ so that 
$$
\lim_{n\to\infty}\norm{T_\G(\xi_n)-\lambda\xi_n}_2 = 0.
$$
As the functions with finite ranges are dense in $L^2(X)$, without loss of generality we can assume each $\xi_n$ has finite range. In particular, we can interpret each $\xi_n$ as a $k$-labeling of the vertices of $\G$ for some $k$ (dependent on $n$). Fixing $n\geq 1$, local-global convergence gives us unit vectors $\xi_{n,m}\in \ell^2(V(G_m))$ for large enough $m$ such that $\text{range}(\xi_{n,m})\sub\text{range}(\xi_{n})$ and
\begin{equation}\label{eq:dvarconv}
\lim_{m\to\infty}d_{\text{var}}(P_{G_m,1}[\xi_{n,m}],P_{\G,1}[\xi_n]) = 0.
\end{equation}
Note that both $P_{G_m,1}[\xi_{n,m}]$ and $P_{\G,1}[\xi_n]$ are probability measures on $\cR\G_{1,k}$, which is a finite set. For every element $(H,o,c)\in \cR\G_{1,k}$, we define
$$
\alpha_{(H,o,c)} = \sum_{\substack{y\in V(H),\\y\neq o}} c(y).
$$
Given any (possibly finite) bounded-degree Borel pmp graph $\G$ on a standard probability space $(X,\mu)$, equipped with a measurable function $\xi:X\to \text{range}(\xi_n)$, we then see that
$$
T_\G(\xi)(x) = \sum_{y\sim x} \xi(y) = \alpha_{(H_x,x,\xi \upharpoonright H_x)},
$$
where $H_x$ is the $1$-neighborhood of $x$ and we view $(H_x, x, \xi \upharpoonright H_x)$ as an element of $\cR\G_{1, k}$. In particular, denoting by $X_{(H,o,c)}$ the set of vertices $x$ of $\G$ such that $(H_x, x, \xi \rest H_x) = (H, o, c)$, we see that 
\begin{align}\label{eq:Hoc}
\begin{split}
    \norm{T_\G(\xi) - \lambda \xi}_2^2 &= \int_X \abs{T_\G(\xi)(x) - \lambda\xi(x)}^2\dmu(x) \\
    &= \int_X\abs{\alpha_{(H_x, x, \xi \upharpoonright H_x)} - \lambda \xi(x)}^2\dmu(x)\\
    &= \sum_{(H,o,c)\in \cR\G_{1,k}} \int_{X_{(H,o,c)}}\abs{\alpha_{(H,o,c)} - \lambda c(o)}^2\dmu(x)\\
    &= \sum_{(H,o,c)\in \cR\G_{1,k}} \mu\big(X_{(H,o,c)}\big) \cdot\abs{\alpha_{(H,o,c)} - \lambda c(o)}^2
\end{split}
\end{align}
Finally, we note that by definition
$$
\mu\big(X_{(H,o,c)}\big) = P_{\G,1}[\xi]((H,o,c)).
$$
From \eqref{eq:dvarconv}, we thus conclude that for every $n$,
\begin{equation}\label{eq:xinmconverges}
\lim_{m\to\infty}\norm{T_{G_m}(\xi_{n,m})-\lambda\xi_{n,m}}_2 = \norm{T_\G(\xi_n)-\lambda\xi_n}_2.
\end{equation}

For notational convenience, we will write $\mu_m$ for the normalized counting measure on $V(G_m)$. To ensure uniform integrability later on, we notice that, for each $t \in \text{range}(\xi_n)$, if we take $A$ to be the singleton set of (the isomorphism class of) the graph of radius $0$ whose single vertex has label $t$, then it follows from~\eqref{eq:dvarconv} that
$$\lim_{m \to \infty} \vert P_{G_m, 1}[\xi_{n, m}](A) - P_{\G, 1}[\xi_n](A) \vert = 0,$$
and so
\begin{equation}\label{eq:kis0}
    \lim_{m\to\infty} \mu_m(\{x\in V(G_m): \xi_{n,m}(x) = t\}) = \mu(\{x\in X: \xi_n(x) = t\}).
\end{equation}
Note that since the $\xi_n$ are uniformly integrable, we can find $\frac{1}{N} > \tilde\delta_N > 0$ for every $N\geq 1$, such that
\begin{equation}\label{eq:xinui}
\forall B\sub X: \mu(B) < \tilde\delta_N \Rightarrow \norm{\xi_n \upharpoonright B}_2^2 < \frac{1}{2N}.
\end{equation}
Let $K_n = \sup\{\abs{t}: t\in \text{range}(\xi_n)\} \geq 1$ and choose $M_n$ so that for all $m\geq M_n$, 
\begin{equation}\label{eq:sumt}
    \sum_{t\in \text{range}(\xi_n)} \abs{\mu_m(\{x\in V(G_m): \xi_{n,m}(x) = t\}) - \mu(\{x\in X: \xi_n(x) = t\})} < \frac{\tilde\delta_n}{2K_n^2}.
\end{equation}
In particular, whenever $A\sub V(G_m)$ for some $m\geq M_n$, we can, using the fact that $\mu$ is nonatomic, find a measurable subset $B\sub X$ such that 
$$
\sum_{t\in \text{range}(\xi_n)} \abs{\mu_m(\{x\in A: \xi_{n,m}(x) = t\}) - \mu(\{x\in B: \xi_n(x) = t\})} < \frac{\tilde\delta_n}{2K_n^2}.
$$

In particular, $\abs{\mu_m(A) - \mu(B)} < \frac{\tilde\delta_n}{2K_n^2} \leq \frac{\tilde\delta_n}{2}$ and 
\begin{align*}
\left| \norm{\xi_{n,m} \upharpoonright A}_2^2 - \norm{\xi_n \upharpoonright B}_2^2 \right| &= \left| \frac{1}{\abs{V(G_m)}} \sum_{x\in A} \abs{\xi_{n,m}(x)}^2 - \int_B \abs{\xi_n(x)}^2 \dmu(x)\right|\\
&= \left| \sum_{t\in \text{range}(\xi_n)} t^2 \big( \mu_m(\{x\in A: \xi_{n,m}(x) = t\})  - \mu(\{x\in B: \xi_n(x) = t\}) \big)\right|\\
&< \frac{\tilde\delta_n}{2}.
\end{align*}

For a given $n\geq 1$, we now define $f_{m_n} = \xi_{n,m_n}$ where $m_n\geq M_n$ is chosen large enough so that, using \eqref{eq:xinmconverges},
$$
\norm{T_{G_{m_n}}(\xi_{n,m_n})-\lambda\xi_{n,m_n}}_2< \norm{T_\G(\xi_n)-\lambda\xi_n}_2 + \frac{1}{n}.
$$
First, since $\norm{T_\G(\xi_n)-\lambda\xi_n}_2 \to 0$, we immediately get
$$
\lim_{n\to\infty}\norm{T_{G_{m_n}}(f_{m_n})-\lambda f_{m_n}}_2 = 0,
$$
as desired. We claim that furthermore the sequence $(f_{m_n})_n$ is uniformly integrable. For this, fix $\eps>0$, and choose $N\in \N$ such that $\frac{1}{N}\leq \eps$. 
Since any finite sequence of square integrable functions is uniformly integrable, we can choose $\delta' > 0$ such that for all $m_n < M_N$:
\begin{equation}\label{eq:smallones}
\forall A\sub V(G_{m_n}): \mu_{m_n}(A) < \delta' \Rightarrow \norm{f_{m_n} \upharpoonright A}_2^2 < \eps.
\end{equation}
Let now $\delta = \min\{\delta', \frac{\tilde\delta_N}{2}\}$, where $\tilde\delta_N < \frac{1}{N}$ is from \eqref{eq:xinui}, and let $A\sub V(G_{m_n})$ for some $m_n$ be a subset of measure $\mu_{m_n}(A) < \delta$. If $m_n < M_n$, \eqref{eq:smallones} immediately gives $\norm{f_{m_n} \upharpoonright A}_2^2 < \eps$. Assume $m_n \geq M_N$. Then by the above, we can find a subset $B\sub X$ of measure 
\begin{equation}\label{eq:measureB}
    \mu(B) < \mu_{m_n}(A) + \frac{\tilde\delta_N}{2} < \tilde\delta_N
\end{equation}
such that
\begin{equation}\label{eq:differencefmnxin}
    \big| \norm{f_{m_n} \upharpoonright A}_2^2 - \norm{\xi_n \upharpoonright B}_2^2 \big| < \frac{\tilde\delta_N}{2} < \frac{1}{2N}.
\end{equation}
Combining \eqref{eq:xinui} and \eqref{eq:measureB}, we get $\norm{\xi_n \upharpoonright B}_2^2 < \frac{1}{2N}$. This together with \eqref{eq:differencefmnxin} then gives
$$
\norm{f_{m_n} \upharpoonright A}_2^2 < \frac{1}{N} \leq \eps.
$$
We conclude that the sequence $(f_{m_n})_n$ is indeed uniformly integrable. Finally, we note that the $f_{m_n}$ we constructed are not necessarily unit vectors. Hence, in order to normalize them without changing the established desired properties, we are left with showing that their norms $\norm{f_{m_n}}_2$ are uniformly bounded away from zero. This follows easily from \eqref{eq:sumt}, which implies that
\begin{align*}
    \big| \norm{f_{m_n}}_2^2 - \norm{\xi_n}_2^2\big| &= \left| \sum_{t\in \text{range}(\xi_n)} t^2 \big(\mu_{m_n}(\{x\in V(G_{m_n}): f_{m_n}(x) = t\}) - \mu(\{x\in X: \xi_n(x) = t\})\big)\right|\\
    &< \frac{\tilde\delta_n}{2} < \frac{1}{2n}.
\end{align*}
Since $\norm{\xi_n}_2 = 1$ for all $n$, we thus immediately conclude that $\norm{f_{m_n}}_2^2 \geq 1 - \frac{1}{2n} \geq \frac{1}{2}$ for all $n\geq 1$. This finishes the proof of (2).

We now prove (1). Assume we have a sequence as in \eqref{eq:uie} for $\lambda$. We will again reduce to considering labelings. Fix $\eps>0$. 
By uniform integrability, we can find a number $M\geq 1$, depending only on $\eps$, such that for every $n\in \N$, there exists a function $\tilde g_n$ on $V(G_{m_n})$ satisfying $\sup_{x\in V(G_{m_n})} \abs{\tilde g_n(x)} \leq M$ and $\norm{f_{m_n} - \tilde g_n}_2 < \frac{\eps}{2}$. 
Fixing a finite $\frac{\eps}{2}$-net $K$ inside the ball of radius $M$ in $\C$, we can now further approximate each $\tilde g_n$ by a function $g_n$ satisfying $\text{range}(g_n)\sub K$ and $\norm{\tilde g_n - g_n}_2 < \frac{\eps}{2}$. In particular, $\norm{f_{m_n} - g_n}_2 < \eps$. Let $k = \abs{K}$. Then we can view every $g_n$ as a $k$-labeling of $G_{m_n}$. 

Let 
$$\eps' = \frac{\eps}{2\cdot M^2\cdot \abs{\cR\G_{1,k}} \cdot (\max_{(H,o,c)\in \cR\G_{1,k}}\{\abs{\alpha_{(H,o,c)} - \lambda c(o)}^2\} +1)}.
$$
By local-global convergence, $Q_{G_{m_n},1,k}$ converges in Hausdorff distance to $Q_{\cG,1,k}$. In particular, there exist $N\in \N$ and, for all $n\geq N$, elements $P_{\G,1}[\xi_n] \in Q_{\cG,1,k}$ associated to some measurable $k$-labelings $\xi_n:X\to K$ such that for all $n\geq N$:
\begin{equation}\label{eq:xin}
    d_{\text{var}}(P_{G_{m_n},1}[g_n], P_{\G,1}[\xi_n]) < \eps'.
\end{equation}
A similar computation as \eqref{eq:Hoc} then gives for all $n\geq N$:
\begin{align*}
    \big| \norm{T_{G_{m_n}}(&g_n) - \lambda g_n}_2^2 - \norm{T_\G(\xi_n) - \lambda \xi_n}_2^2 \big| \\
    &= \left|\sum_{(H,o,c)\in \cR\G_{1,k}} \left( P_{G_{m_n},1}[g_n]((H,o,c)) - P_{\G,1}[\xi_n]((H,o,c)) \right) \abs{\alpha_{(H,o,c)} - \lambda c(o)}^2 \right|\\
    &\leq \abs{\cR\G_{1,k}}\cdot \max_{(H,o,c)\in \cR\G_{1,k}}\{\abs{\alpha_{(H,o,c)} - \lambda c(o)}^2\} \cdot d_{\text{var}}(P_{G_{m_n},1}[g_n], P_{\G,1}[\xi_n])\\
    &< \frac{\eps}{2}.
\end{align*}
In particular, since $\norm{T_{G_{m_n}}(g_n) - \lambda g_n}_2 \to 0$ as $n\to\infty$, we can choose $n$ large enough so that
$$
\norm{T_\G(\xi_n) - \lambda \xi_n}_2^2 < \eps.
$$
In order to conclude that $\lambda\in\sigma(T_\G)$, the only remaining thing to check is that for the sequence $(\xi_n)_n$ we constructed, the norms $\norm{\xi_n}_2$ are uniformly bounded away from zero. The argument is similar to the second half of the proof of (2). Fix $n\in \N$, and note that \eqref{eq:xin} for graphs of radius 0 implies that for all $t\in K$:
$$
\mu_{m_n}(\{x\in V(G_{m_n}): g_n(x) = t\}) - \mu(\{x\in X: \xi_n(x) = t\}) < \eps'.
$$
In particular,
\begin{align*}
    \big|\norm{g_n}_2^2 - \norm{\xi_n}_2^2\big| &= \big|\sum_{t\in K} t^2\cdot \big(\mu_{m_n}(\{x\in V(G_{m_n}): g_n(x) = t\}) - \mu(\{x\in X: \xi_n(x) = t\})\big)\big|\\
    &< \abs{K}\cdot M^2 \cdot \eps'\\
    &< \eps.
\end{align*}
Here we used the obvious fact that $\abs{K} = k \leq \abs{\cR\G_{1,k}}$. Since $\norm{f_{m_n}}_2 = 1$ for all $n$ and $\norm{f_{m_n} - g_n}_2 < \eps$, we conclude that for $\eps < \frac{1}{10}$,
$$
\norm{\xi_n}_2^2 \geq \norm{g_n}_2^2 - \eps \geq (1-\eps)^2 - \eps \geq \frac{1}{2}.
$$
This finishes the proof of (1), and thus of the theorem.
\end{proof}

Theorem~\ref{thm: locglob spectrum} allows one to exactly compute the spectrum of many bounded-degree Borel pmp graphs of interest. For instance, an easy concrete example where we can use it is irrational rotation:

\begin{exmp}
\label{exmp:rotationspec}
    Let $\theta$ be an irrational number, and let $\G$ be the graph on $[0, 1)$ obtained by placing an edge between vertices $x$ and $y$ if and only if $x = y \pm \theta \mod{1}$. Then $\G$ arises as a local-global limit of the sequence $(G_n)_{n \geq 3}$, where for each $n \geq 3$, $G_n$ is the cyclic graph on $n$ vertices (see for instance \cite[Section~12]{HLS14}). It is easy to see that the spectrum of $G_n$ is 
    $$\{ 2\cos\Big(\frac{2\pi k}{n}\Big) : k \in \N, 0 \leq k \leq n - 1\},$$
    the set of $x$-coordinates of the $n$th roots of unity. Moreover, these eigenvalues are witnessed by the eigenvectors
    $$
    v_k = \begin{pmatrix}
        1\\
        \omega^k\\
        \omega^{2k}\\
        \vdots\\
        \omega^{(n-1)k}
    \end{pmatrix}
    $$
    for $0\leq k\leq n-1$, where $\omega = e^{2\pi i/n}$. Therefore, by Theorem~\ref{thm: locglob spectrum}, the spectrum of $\G$ is the entire interval $[-2, 2]$. 
\end{exmp}

Another example concerns \textit{spectral gap}, which is known to be preserved under local-global equivalence (see \cite[Section~11]{HLS14}). Assume now that $(G_n)_n$ is a sequence of finite graphs that converges locally-globally to a bounded-degree Borel pmp graph $\G$ admitting uniformly integrable eigenfunctions. Then if each $G_n$ has spectral gap at least $\eps>0$, then also $\cG$ has spectral gap at least $\eps$. This follows easily from Theorem~\ref{thm: locglob spectrum}, together with the fact that for a self-adjoint matrix $T$ and a positive $\eps_0$, $\norm{Tf - \lambda f}_2 < \eps_0 \norm{f}_2$ implies that $\lambda$ is at most $\eps_0$ away from an eigenvalue of $T$. The other direction is false in general. Indeed, similar to Example~\ref{eg:locglob}, one can add a graph without spectral gap with increasingly small mass to a sequence of graphs with spectral gap bounded away from $0$ to obtain a local-global limit that has spectral gap.

On a related note, we remark that the \emph{cost} of the limit graph can also be computed given the costs of the approximating graphs; see \cite{at2020,CGdlS21}. Finally, we point out that the spectra of ($d$-regular) \emph{random graphs} have been the subject of extensive study. While several questions concerning the local-global convergence of sequences of such graphs remain open, many (partial) results are known; see, for instance, \cite{HLS14, BS18, BS19, BS22}.

\bibliographystyle{amsalpha}
\bibliography{references}

\end{document}